\documentclass[11pt,a4]{amsart}

\usepackage[utf8]{inputenx}
\usepackage{lmodern}
\usepackage{textcomp}
\usepackage[english]{babel}
\usepackage[T1]{fontenc}
\usepackage{bm}
\usepackage{amssymb,amsthm,amsmath}
\usepackage{indentfirst}
\usepackage{color}
\usepackage{graphicx}
\usepackage{enumerate}
\usepackage{caption}
\usepackage{empheq}
\usepackage{xcolor}
\usepackage{color}
\usepackage{amssymb}
\usepackage{booktabs}
\usepackage{graphicx}
\usepackage[font=small,labelfont=bf,labelsep=period,tableposition=top]{caption}
\usepackage{hyperref}

\newtheorem{theorem}{Theorem}[section]
\newtheorem{lemma}[theorem]{Lemma}
\newtheorem{proposition}[theorem]{Proposition}

\newtheorem{definition}{Definition}[section]
\newtheorem{corollary}[theorem]{Corollary}

\usepackage{ifthen}
\provideboolean{shownotes}\setboolean{shownotes}{true}
\newcommand{\margnote}[1]{\ifthenelse{\boolean{shownotes}}
	{\marginpar{\raggedright\tiny\texttt{#1}}}{}}

\hypersetup
{   bookmarks=false,
	colorlinks=true,
	urlcolor=blue,
	pdfauthor = {Andrea Aspri},
	pdftitle = {Harmonic functions with cavity},
	pdfcreator = {LaTeX, TeXmaker, pdfLaTeX, Hyperref}
}
	
\begin{document}
\baselineskip=15pt

\begin{center}\sf
{\Large Analysis of a Mogi-type model describing surface deformations}\vskip.2cm
{\Large induced by a magma chamber embedded in an elastic half-space}\vskip.40cm
Andrea ASPRI\footnote{Dipartimento di Matematica ``G. Castelnuovo'',
	Sapienza -- Universit\`a di Roma, P.le Aldo Moro, 2 - 00185 Roma (ITALY), \texttt{\tiny aspri@mat.uniroma1.it}},
Elena BERETTA\footnote{Dipartimento di Matematica,
	Politecnico di Milano, Via Edoardo Bonardi - 20133 Milano (ITALY), \texttt{\tiny elena.beretta@polimi.it}},
Corrado MASCIA\footnote{Dipartimento di Matematica ``G. Castelnuovo'',
	Sapienza -- Universit\`a di Roma, P.le Aldo Moro, 2 - 00185 Roma (ITALY), \texttt{\tiny mascia@mat.uniroma1.it}}
\end{center}
\vskip.5cm

\begin{quote}\footnotesize 
{\sf Abstract.} Motivated by a vulcanological problem, we establish a sound mathematical approach for surface deformation effects
generated by a magma chamber embedded into Earth's interior and exerting on it a uniform hydrostatic pressure.
Modeling assumptions translate the problem into classical elasto-static system (homogeneous and isotropic) in an 
half-space with an embedded cavity.
The boundary conditions are traction-free for the air/crust boundary and uniformly hydrostatic for the crust/chamber boundary.
These are complemented with zero-displacement condition at infinity (with decay rate).

After a short presentation of the model and of its geophysical interest, we establish the well-posedness of the problem and
provide an appropriate integral formulation for its solution for cavity with general shape.
Based on that, assuming that the chamber is centred at some fixed point $\bm{z}$ and has diameter $r>0$, small with respect to the depth $d$, we derive rigorously the principal term in the asymptotic expansion for the surface deformation as $\varepsilon=r/d\to 0^+$.
Such formula provides a rigorous proof of the Mogi point source model in the case of spherical cavities generalizing it to the case of cavities of arbitrary shape. 
\end{quote}
\vskip.15cm

\vskip.15cm


\section{Introduction}
Measurements of crustal deformations are crucial in studying and monitoring volcanoes activity.
In this context, one of the main goals is to investigate the mechanics of volcanic systems in order to obtain models for which
the predicted synthetic data best fit the observed ones.
Indeed, if the model is accurate it gives good images of the Earth's interior.
In particular, concerning volcanoes monitoring, synthetic data might be used to localize underground magma chambers and their volume changes.
In this sense, a well-established and widely used model, in the geological literature, is the one proposed by Mogi in \cite{Mogi58} where the magma
chamber is modelled by a pressurized spherical cavity of radius $r$, buried in a homogeneous, isotropic, elastic half-space $(\mathbb{R}^3_-)$ at depth $d\ll r$. 
Specifically, Mogi provides an explicit formula of the leading term of the displacement field on the boundary $(\mathbb{R}^2)$ of the half-space when the ratio $r/d$ goes to zero.
This model has been further investigated by McTigue in \cite{McTigue87} who extends the analysis by detecting the second order term
in the asymptotic formula.
For other geometries, we mention \cite{Davis86} where Davies considers the case of a pressurized cavity in the shape of prolate and
oblate ellipsoids, restricted to radius ratio to depth sufficiently small, deriving, also in this case, an explicit formula of the leading order term.

We stress that the analysis carried out in \cite{Davis86,McTigue87,Mogi58} and the derivation of the formulas therein are formal.
For this reason the main purpose of this paper is to cast and generalize these results in a rigorous mathematical framework. 
More precisely, we derive a rigorous asymptotic expansion of the boundary displacement vector field due to the presence
of a cavity $C$, buried in an isotropic, homogeneous and elastic half-space $(\mathbb{R}^3_-)$, of the form
\begin{equation*}
	C=C_\varepsilon:=\bm{z}+\varepsilon \varOmega
\end{equation*}
where $\bm{z}$ is the center of the cavity, $\varepsilon$ is a suitable scaling parameter (see Section 2) and $\varOmega$, shape of the cavity, is a bounded Lipschitz domain
containing the origin.
Furthermore, we follow the assumption in \cite{Davis86,Mogi58} supposing that the boundary of the cavity $C$
is subjected to a constant pressure $p$. 
In details, denoting by $\bm{u}_\varepsilon$  the displacement vector field, we have, for $\bm{y}\in \mathbb{R}^2$,
\begin{equation}\label{mainformula}
	u^{k}_{\varepsilon}(\bm{y})=\varepsilon^3|\varOmega| p\widehat{\nabla}_{\bm{z}}\bm{N}^{(k)}(\bm{z},\bm{y}):\mathbb{M}\mathbf{I}+O(\varepsilon^4),
\end{equation}
for $k=1,2,3$, as $\varepsilon\rightarrow 0$ (see Theorem \ref{th asymp exp}), where $u^k_{\varepsilon}$ stands for the $k$-th
component of the displacement vector.
Here $p\,\varepsilon^3$ represents the total work exerted by the cavity on the half-space, $\widehat{\nabla}\bm{N}^{(k)}$ denotes the symmetric part of the gradient related to the $k$-th column vector of the {\it Neumann function} that is the solution to
\begin{equation*}
	\textrm{div}(\mathbb{C} \widehat{\nabla}\mathbf{ N}(\cdot,\bm{y}))=\delta_{\bm{y}}\mathbf{I}
	\quad\textrm{in }\, \mathbb{R}^3_-,\qquad\qquad
        \partial \mathbf{N}(\cdot,\bm{y})/\partial\bm{\nu}=0\quad  \textrm{on }\, \mathbb{R}^2,
\end{equation*}
with the decay conditions at infinity 
\begin{equation*}
	\mathbf{N}=O(|\bm{x}|^{-1}),\qquad\qquad |\nabla\mathbf{N}|=O(|\bm{x}|^{-2}),
\end{equation*}
where $\partial/\partial\bm{\nu}$ is the conormal derivative operator. 
In addition, $\mathbf{I}$ is the identity matrix in $\mathbb{R}^3$ and $\mathbb{M}$ is the fourth-order moment elastic tensor defined by
\begin{equation*}
\mathbb{M}:=\mathbb{I}+\frac{1}{|\varOmega|}\int\limits_{\partial\varOmega}\mathbb{C}(\bm{\theta}^{qr}(\bm{\zeta})\otimes \bm{n}(\bm{\zeta}))\, d\sigma(\bm{\zeta}),
\end{equation*}
for $q,r=1,2,3$, where $\mathbb{I}$ is the symmetric identity tensor, 
$\mathbb{C}$ is the isotropic elasticity tensor and $\bm{n}$
is the outward unit normal vector to $\partial\varOmega$.\\
Finally, the functions $\bm{\theta}^{qr}$, with $q,r=1,2,3$, are the solutions to
\begin{equation*}
\textrm{div}(\mathbb{C}\widehat{\nabla}\bm{\theta}^{qr})=0\quad \textrm{in}\, \mathbb{R}^3\setminus \varOmega,\qquad\quad
\frac{\partial \bm{\theta}^{qr}}{\partial\bm{\nu}}=-\frac{1}{3\lambda+2\mu}\mathbb{C}\bm{n}\quad \textrm{on}\, \partial\varOmega,
\end{equation*}
with the decay conditions at infinity
\begin{equation*}
|\bm{\theta}^{qr}|=O(|\bm{x}|^{-1}),\qquad\quad |\nabla\bm{\theta}^{qr}|=O(|\bm{x}|^{-2}),\qquad \textrm{as}\, |\bm{x}|\to \infty.
\end{equation*} 

In order to establish these results we use the approach introduced by Ammari and Kang, see for example \cite{Ammari-libroelasticita,Ammari-Kang,Ammari-Kang1}, based on layer potentials techniques and following the path outlined in \cite{AsprBereMasc}. Despite the difficulty to deal with a boundary value problem for an elliptic system in a half-space,
see for example \cite{Amrouche-Dambrine-Raudin}, this strategy allows us to prove the existence and uniqueness of the displacement field generated by an arbitrary finite cavity $C$ and then to derive the asymptotic expansion \eqref{mainformula} as $\varepsilon$ goes to zero.
The leading term in \eqref{mainformula} contains the information on the location of the cavity, $\bm{z}$, on its shape and on the work $p\,\varepsilon^3$, hence it might be used to detect the unknown cavity if boundary measurements of the displacement field are given.

The approach based on the determination of asymptotic expansions in the case of small inclusions or cavities in bounded domains, which goes back to Friedman and Vogelius \cite{Friedman-Vogelius}, has been extensively and successfully used for the reconstruction, from boundary measurements, of location and geometrical features of unknown conductivity inhomogeneities in the framework of electrical impedence tomography (see \cite{Ammari-Kang} for an extended bibliography). 
In the last decade some of these results have been extended to elasticity in the case of bounded domains. Specifically, starting from boundary measurements given by the couple potentials/currents or deformations/tractions, in the case, respectively, of electrical impedence tomography and linear elasticity, information about the conductivity profile and the elastic parameters of the medium have been inferred.   
It is well known that without any a priori assumptions on the features of the problem, the reconstruction procedures give poor quality results.
This is due to the severe ill-posedness of the inverse boundary value problem modelling both the electrical impedance tomography \cite{AlessBerRossVess} and the elasticity problems \cite{AlessDiCristoMorasRoss,Morassi-Rosset}. 
However, in certain situations one has some a priori information about the structure of the medium to be reconstructed.
These additional details allow to restore the well-posedness of the problem and, in particular, to gain uniqueness and Lipschitz continuous dependence of inclusions or cavities from the boundary measurements. 
One way to proceed, for instance, is to consider the medium with a smooth background conductivity or elastic parameters except for a finite number of small inhomogeneities \cite{Ammari-libroelasticita,Ammari-Kang}.
Therefore, by means of partial or complete asymptotic formulas of solutions to the conductivity/elastic problems,
information about the location and size of the inclusions  
can be reconstructed.
In the last decade, based on this strategy, some efficient algorithms that allow to implement valid reconstruction procedures have been proposed, see for example \cite{Ammari-libroelasticita,Ammari-Kang,Guzina-Bonnet}.
 
The paper is organized as follows. 
In Section 2, we illustrate the model of ground deformations and motivate the assumptions at the basis of the simplified linearized version.
In Section 3, we recall some arguments about linear elasticity and layer potentials techniques and then we analyze the well-posedness
of the direct problem via an integral representation formula for the displacement field.
Section 4 is devoted to the proof of the main result regarding the asymptotic formula for the boundary diplacement field.
In addition, as a consequence of the asymptotic expansion, we obtain the classical Mogi's formula 
for spherical cavities.
 
\subsection*{Notations}
{\small
We denote scalar quantities in italic type, e.g. $\lambda, \mu, \nu$,
points and vectors in bold italic type, e.g.  $\bm{x}, \bm{y}, \bm{z}$ and $\bm{u}, \bm{v}, \bm{w}$, 
matrices and second-order tensors in bold type, e.g.  $\mathbf{A}, \mathbf{B}, \mathbf{C}$, 
and fourth-order tensors in blackboard bold type, e.g.  $\mathbb{A}, \mathbb{B}, \mathbb{C}$. 
The half-space $\{\bm{x}=(x_1,x_2,x_3)\in \mathbb{R}^3\,:\,x_3\leq 0\}$ is indicated by $\mathbb{R}^3_-$.
The unit outer normal vector to a surface is represented by $\bm{n}$.
The transpose of a second-order tensor $\mathbf{A}$ is denoted by $\mathbf{A}^T$
and its symmetric part by $\widehat{\mathbf{A}}=\tfrac{1}{2}\left(\mathbf{A}+\mathbf{A}^T\right)$.
To indicate the inner product between two vectors $\bm{u}$ and $\bm{v}$ we use
$\bm{u}\cdot \bm{v}=\sum_{i} u_{i} v_{i}$
whereas for second-order tensors $\mathbf{A}:\mathbf{B}=\sum_{i,j}a_{ij} b_{ij}$. The cross product of two vectors $\bm{u}$ and $\bm{v}$ is denoted by $\bm{u}\times \bm{v}$. With $|\mathbf{A}|$ we mean the norm induced by the inner product between second-order tensors, that is $|\mathbf{A}|=\sqrt{\mathbf{A}:\mathbf{A}}$.
}

\section{Description of the mathematical model}
Monitoring of volcanoes activity, targeted to the forecasting of volcanic hazards and development of
appropriate prevention strategies, is usually performed by combining different types of geophysical measurements. 
Ground deformations are among the most significant being directly available (in particular, with the development
of Global Positioning Systems) and, at the same time, straightly interpretable in term of elastic behaviors of the
Earth's crust adjacent to the magma chamber (see \cite{Dzurisin06, Segall13}).
A well-established model has been proposed by Mogi, \cite{Mogi58}, following previous results
(see description in \cite{DeNaPing96, Lisowski06, Segall10}),
based on the assumption that ground deformation effects are primarily generated by the
presence of an underground magma chamber exerting a uniform pressure on the surrounding medium.
Precisely, the model relies on three key founding schematisations\par
{\sl 1. Geometry of the model.} The earth's crust is an infinite half-space (with free air/crust surface located on
the plane $x_3=0$) and the magma chamber, buried in the half-space, is assumed to be spherical
with radius $r$ and depth $d$ such that $r\ll d$.\par
{\sl 2. Geophysics of the crust.} The crust is a perfectly elastic body, isotropic and homogeneous, whose deformations are
described by the linearized elastostatic equations, hence are completely characterized by the Lam\'e parameters
$\mu,\lambda$ (or, equivalently, Poisson ratio $\nu$ and shear modulus $\mu$).
The free air/crust boundary is assumed to be a traction-free surface.\par
{\sl 3. Crust-chamber interaction.} The cavity describing the magma chamber is assumed to be filled with an ideal 
incompressible fluid at equilibrium, so that the pressure $p$ exterted on its boundary on the external elastic medium
is hydrostatic and uniform.

Assuming that the center of the sphere is located at $\bm{z}=(z_1,z_2,z_3)$ with $z_3<0$,
the displacement $\bm{u}=(u_1,u_2,u_3)$ at a surface point $\bm{y}=(y_1,y_2,0)$ is given by
\begin{equation}\label{mogi}
	u^{\alpha}(\bm{y})=\frac{1-\nu}{\mu}\frac{\varepsilon^3p (z_{\alpha}-y_{\alpha})}{|\bm{z}-\bm{y}|^3},\quad (\alpha=1,2)\qquad
	u^{3}(\bm{y})=\frac{1-\nu}{\mu}\frac{\varepsilon^3p\,z_3}{|\bm{z}-\bm{y}|^3}
\end{equation}
in the limit $\varepsilon:=r/|z_3|\to 0$ (see Fig. \ref{fig:mogi}).
A higher-order approximation has been proposed by McTigue \cite{McTigue87} with the intent of providing
a formal expansion able to cover  the case of a spherical body with finite (but small) positive radius.

\begin{figure}[hbt]
\includegraphics[width=12cm]{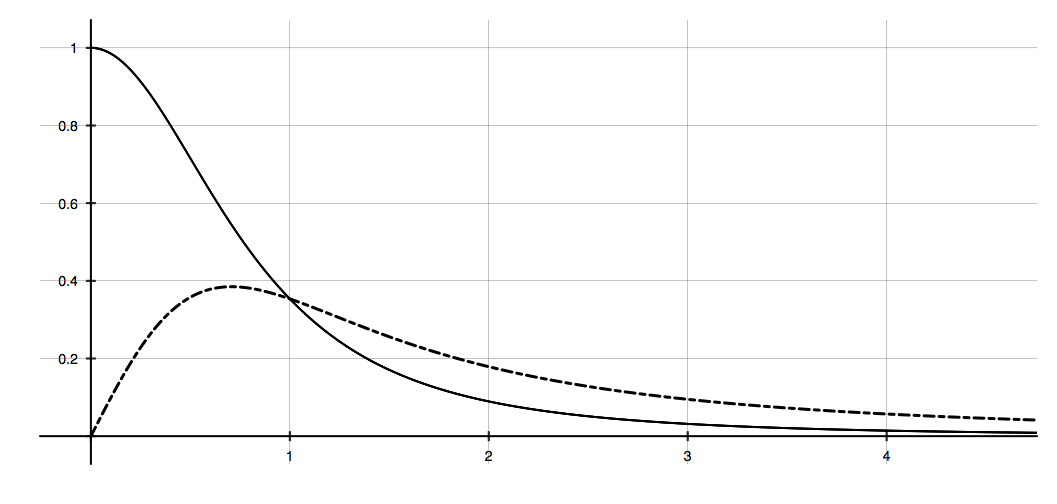}
\caption{\footnotesize Normalized Mogi displacement profiles given in \eqref{mogi}:
horizontal components $u_\alpha$, $\alpha=1,2$, dashed line;
vertical component $u_3$, continuous line.}
\label{fig:mogi}
\end{figure}

Being based on the assumption that the ratio radius/depth $\varepsilon:=r/d$ is small, the Mogi model corresponds to the 
assumption that the magma chamber is well-approximated by a single point producing a uniform pressure 
in the radial direction; as such, it is sometimes referred to as a {\it point source model}.
However, even if the source is reduced to a single point, the model still records the spherical form of the cavity.
Different geometrical form may lead to different deformation effects (as will be clear in the subsequent analysis).

The Mogi model has been widely applied to real data of different volcanoes to infer approximate
location and strength of the magma chamber.
The main benefit of such strategy lies in the fact that it provides a simple formula with an explicit
appearance of the basic physical parameters depth and total strength (combining pressure and volume)
and, thus, that it can be readily compared with real deformation data to provide explicit forecasts.

The simplicity of formulas \eqref{mogi} makes the application model viable,
but it compensates only partially the intrinsic reductions of the approach.
As a consequence, variations of the basic assumptions have been proposed to provide more realistic frameworks.
With no claim of completeness, we list here some generalizations of the Mogi model available in the literature.

Since real data sets often exhibit deviations from radially symmetric deformations, different shapes for
the point source have been proposed. 
Guided by the request of furnishing explicit formulas, such attempt has primarily focused on
ellipsoidal geometries and, in particular, on oblate and prolate ellipsoids \cite{Davis86,YangDavis88}.
With respect to the spherical case, these new configurations are able to indicate the presence of some 
elongations of the chamber and possible tilt with respect to the surface.
As a drawback, formulas for ellipsoid cavities turn to be rather complicated, involving, in the general case,
elliptic integrals.

Different configurations, such as rectangular dislocations (see \cite{Okad85}) and horizontal penny-shaped cracks
(see \cite{FialKhazSimo01}), have also been considered still with the target of furnishing an explicit formula for ground
deformation to be compared with real data by means of appropriate inversion algorithms.
Still relative to the geometry of the model, efforts have also been directed to the case of a non-flat crust surface,
with the intent of taking into account the specific topography, as given by the local elevation above mean sea level
of the region under observation \cite{WillWadg98, CurrDelNGanc08}.

Studies have been addressed to a finer description of the geophysical properties of the crust,
with special attention to the case of heterogeneous rheologies.
Indeed, different parts of the crust may exhibit different mechanical properties due to the presence of stiff
(lava flows, welded pyroclasts, intrusions) and soft layers (non-welded pyroclasts, sediments),
see \cite{MancWaltAmel07} and references therein.
Variations of elastic parameters may also arise as consequence of the thermic properties of the magma inside
the chamber, which determines different local properties in the neighborhood of the cavity, so that the presence
of an additional layer surrounding the chamber could be appropriated.
Additionally, nonuniform pressure distribution on the boundary of the chamber may arise, as an example,
from a nonuniform nature of the material filling the cavity (see discussion in \cite{DeNaPing96}).
Incidentally, we stress that the use of nonlinear elastic models in this area is still in a germinal phase
and it would require a more circumstantial regard.

For completeness, we mention also the attempts of combining elastic properties with gravitational effects and
time-dependent processes modeling of the crust by means of elasto-dynamic equations or viscoelastic rheologies,
(among others, see \cite{BGCF08, CBDSB10}).

In all cases, refined descriptions have the inherent drawback of requiring a detailed knowledge of the crust elastic properties.
In absence of reliable complete data and measurements, the risk of introducing an additional degree of freedom
in the parameter choice is substantial.
This observation partly supports the approach of the Mogi model which consists in keeping as far as possible
the parameters choice limited and, consequently, the model simple.

As the above overview shows, the geological literature on the topic is extensive.
On the contrary, the mathematical contributions seem to be still lacking. 
The principal aim of this article is to contribute with a rigorous mathematical analysis  of such modelling,
focusing on the most basic model (the Mogi model) with the crucial difference relative to the shape of the cavity,
which will be not forced to be neither a sphere nor an ellipsoid, but a general subdomain of the half-space.
Possibly, more refined models for the same phenomenon will be object of forthcoming research.

Now, let us introduce in details the boundary value problem which emerges from the previous assumptions
on the geometry of the model, geophysics of the crust and crust-chamber interaction.
Denoting by $\mathbb{R}^3_-$ the (open) half-space described by the condition $x_3<0$,
the domain $D$ occupied by the Earth's crust is $D:=\mathbb{R}^3_-\setminus \overline{C}$,
where $C\subset\mathbb{R}^3_-$, describing the magma chamber, is assumed to be an open set with 
bounded boundary $\partial C$.
Hence, the boundary of $D$ is composed by two disconnected components:
the two-dimensional plane $\mathbb{R}^2:=\{\bm{y}=(y_1,y_2,y_3)\in \mathbb{R}^3\,:\,y_3=0\}$,
which constitutes the free air/crust border, and the set $\partial C$, corresponding to the crust/chamber edge.

Next, we introduce the elastic description of the medium filling the domain $D$.
We follow the general approach presented in \cite{Ciar88}, valid for nonlinear elasticity,
postponing at the very end the linear approximation, which will be considered in the rest of the present study.
Analyzing the nonlinear version of the model is an intriguing topic which is of course much more complicated
(in particular, for the inference of asymptotic deformation formulas) and is left for future investigations.

Assuming the medium to be homogeneous, in absence of volume forces and with a uniform hydrostatic
pressure $p$ exerted by the chamber on $\partial C$, the deformation $\bm{\varphi}$ of the medium satisfies
\begin{equation*}
	\text{div}\bigl(\nabla\bm{\varphi}\,\mathbf{\Sigma}\bigr)=\bm{0}\quad \textrm{in }D,\qquad
	\bigl(\nabla\bm{\varphi}\,\mathbf{\Sigma}\bigr)\bm{e}_3=\bm{0}\quad \textrm{in }\mathbb{R}^2,\qquad
	\bigl(\nabla\bm{\varphi}\,\mathbf{\Sigma}\bigr)\bm{n}=p\,\bm{n}\quad \textrm{in }\partial C,\qquad
\end{equation*}
where $\mathbf{\Sigma}$ is the second Piola--Kirchhoff stress tensor 
and $\bm{n}$ is the exterior normal vector to $\partial C$ (see Theorem 2.6.2 in \cite{Ciar88}).
In addition, appropriate far-field condition on $\bm{\varphi}$ has to be assigned.

For {\it homogeneous elastic materials}, an appropriate constitutive law giving the stress tensor $\mathbf{\Sigma}$
in term of $\nabla\bm{\varphi}$ is provided.
Requiring in addition the {\it axiom of material frame-indifference}, the stress tensor turns out to depend
on $\nabla\bm{\varphi}$ by means of the Green--Cauchy strain tensor $\nabla\bm{\varphi}^T\nabla\bm{\varphi}$ only
(for details, see \cite{Ciar88}).
Introducing the {\it displacement vector} $\bm{u}:=\bm{\varphi}-\textrm{\bf id}$, where $\textrm{\bf id}$ denotes the identity map, 
the tensor $\nabla\bm{\varphi}^T\nabla\bm{\varphi}$ can be rewritten as $\nabla\bm{\varphi}^T\nabla\bm{\varphi}=\mathbf{I}+2\mathbf{E}$
where $\mathbf{E}$, called the Green--SaintVenant stress tensor, is given by
\begin{equation*}
	\mathbf{E}:=\widehat{\nabla}\bm{u}+\frac12\nabla\bm{u}^T\nabla\bm{u}
\end{equation*}
(here $\widehat{\nabla}\bm{u}$ denotes the symmetric part of $\nabla\bm{u}$).
Correspondingly, the stress--strain constitutive relation has the form $\mathbf{\Sigma}=\Sigma(\mathbf{E})$,
with the function $\Sigma$ determined by further modeling assumptions, and the displacement $\bm{u}$
satisfies the (nonlinear) equation
\begin{equation*}
	\text{div}\bigl\{\bigl(\mathbf{I}+\nabla\bm{u}\bigr)\,\Sigma(\mathbf{E})\bigr)=\bm{0}\quad \textrm{in }D.
\end{equation*}
since $\nabla\bm{\varphi}=\mathbf{I}+\nabla\bm{u}$.
For isotropic media and under the assumption that the reference configuration $\bm{u}=\bm{0}$ is stress-free,
i.e. if $\Sigma(\mathbf{O})=\mathbf{O}$ where $\mathbf{O}$ is the zero matrix, the response function $\Sigma$ takes the special form
\begin{equation*}
	\Sigma(\mathbf{E})=\lambda(\textrm{tr} \mathbf{E})\mathbf{I}+2\mu \mathbf{E}+o(\mathbf{E})
	\quad\textrm{as }\mathbf{E}\to \mathbf{O}
\end{equation*}
where $\lambda,\mu$ are two arbitrary constants, known as {\it Lam\'e constants},
which depend on the specific elastic material under consideration
(and $o$ is the standard Landau symbol).
It is common to use the {\it Poisson ratio} $\nu$ which is related to $\lambda$ and $\mu$ by the identity
\begin{equation*}
	\nu=\frac{\lambda}{2(\lambda+\mu)}.
\end{equation*}
At this point, we perform some final simplications dictated by a linearization argument.
Precisely, restricting the attention to the regime $\mathbf{E}\to \mathbf{O}$ and $\nabla\bm{u}\to \mathbf{O}$,
we consider the approximations
\begin{equation*}
	\Sigma(\mathbf{E})\approx \mathbb{C}\mathbf{E},
	\qquad \mathbf{E}\approx \widehat{\nabla}\bm{u}.
	\qquad  \mathbf{I}+\nabla\bm{u}\approx \mathbf{I}
\end{equation*}
where $\mathbb{C}:=\lambda \mathbf{I}\otimes \mathbf{I}+2\mu\mathbb{I}$ is the fourth-order isotropic elasticity tensor
with $\mathbb{I}$  fourth-order tensor defined by $\mathbb{I} \mathbf{A}:=\widehat{\mathbf{A}}$.

So, we end up with the following boundary value problem
\begin{equation*}
	\text{div}\bigl(\mathbb{C}\widehat{\nabla}\bm{u}\bigr)=\bm{0}\quad \textrm{in }D,\qquad
	\bigl(\mathbb{C}\widehat{\nabla}\bm{u}\bigr)\bm{e}_3=\bm{0}\quad \textrm{in }\mathbb{R}^2,\qquad
	\bigl(\mathbb{C}\widehat{\nabla}\bm{u}\bigr)\bm{n}=p\,\bm{n}\quad \textrm{in }\partial C,\qquad
\end{equation*}
which we consider with the asymptotic conditions at infinity
\begin{equation*}
	\lim_{|\bm{x}|\to\infty}\bm{u}(\bm{x})
	=\lim_{|\bm{x}|\to\infty}|\bm{x}|\nabla \boldsymbol{u}(\bm{x})=\bm{0}.
\end{equation*}
The conjuction of conditions on $\partial D$ and at infinity determines a displacement-traction problem.

At this point, the model provides the displacement $\bm{u}$ of a generic \underline{finite} cavity $C$.
The next step is to deduce a corresponding point source model, in the spirit of the Mogi spherical one.
To this aim, we assume the cavity $C$ of the form
\begin{equation*}
	C=d\bm{z}+r \varOmega
\end{equation*} 
where $d, r>0$ are charateristic length-scales for depth and diameter of the cavity,  its center $d\bm{z}$ belongs
to $\mathbb{R}^3_-$ and its shape $\varOmega$ is a bounded domain containing the origin. 
The Mogi model corresponds to $\varOmega$ given by a sphere of radius $1$.

Introducing the rescaling $(\bm{x},\bm{u})\mapsto (\bm{x}/d,\bm{u}/r)$ and denoting the new variables
again by $\bm{x}$ and $\bm{u}$, the above problem takes the form (with unchanged far-field conditions)
\begin{equation*}
	\text{div}\bigl(\mathbb{C}\widehat{\nabla}\bm{u}\bigr)=\bm{0}\quad \textrm{in }D_\varepsilon,\qquad
	\bigl(\mathbb{C}\widehat{\nabla}\bm{u}\bigr)\bm{e}_3=\bm{0}\quad \textrm{in }\mathbb{R}^2,\qquad
	\bigl(\mathbb{C}\widehat{\nabla}\bm{u}\bigr)\bm{n}=p\,\bm{n}\quad \textrm{in }\partial C_\varepsilon,\qquad
\end{equation*}
where $\varepsilon=r/d$, $C_\varepsilon:=\bm{z}+\varepsilon\varOmega$,
$D_\varepsilon:=\mathbb{R}^3_-\setminus \overline{C}_\varepsilon$
and $p$ is a ``rescaled'' pressure, ratio of the original pressure $p$ and $\varepsilon$.

Denoting by $\bm{u}_\varepsilon$ the solution to such boundary value problem, the point reduction consists
in considering the limiting behavior $\varepsilon\to 0$ of $\bm{u}_\varepsilon$ and, precisely, in determining 
an asymptotic expansion valid for $\bm{y}\in\mathbb{R}^2$ with the form
\begin{equation}\label{expansion}
	\bm{u}_\varepsilon(\bm{y})=\varepsilon^\alpha p\, \bm{U}(\bm{z},\bm{y})+o(\varepsilon^\alpha)
	\qquad\textrm{as }\varepsilon\to 0^+\qquad (\bm{y}\in\mathbb{R}^2)
\end{equation}
for some appropriate exponent $\alpha>0$ and principal term $\bm{U}$.
In the special case of $\varOmega$ given by a sphere of radius $1$, the principal term $\bm{U}$
is expected to coincide with \eqref{mogi} for $r=\varepsilon$.
Incidentally, we observe that, being the expansion \eqref{expansion} valid for small deformations,
the expression of the principal term $\bm{U}$ is expected to be correct also in the nonlinear case when $\mathbf{\Sigma}=\Sigma(\mathbf{E})$ is sufficiently smooth.

The multiplicative factor $p$ emerges as consequence of the linearity of the problem.
Since the exponent $\alpha$ turns out to be equal to $3$, it is readily seen that the expansion does not separate
the information relative to cavity volume and cavity pressure, and it is capable, at most, of determining the
total work exerted by the magma chamber on the surrounding medium.
Uncoupling of the two contributions can be obtained only at the price of considering higher order terms in the expansion (as performed in \cite{McTigue87} in the spherical case).

With the purpose of developing a rigorous mathematical approach for the problem, the subsequent sections are addressed to \par
{\bf i.} establishing an integral formulation for the linear problem satisfied by  $\bm{u}$;\par
{\bf ii.} proving well-posedness for general finite cavities, based on such a formulation;\par
{\bf iii.} validating the expansion \eqref{expansion} with $\alpha=3$ and an explicit principal term $\bm{U}$.\\
We highlight the fact that, in what follows, there is no specific restriction in the choice of the geometry of the cavity (except
for some minimimal regularity assumptions), so that we are able to extend the known formulas for very specific shapes
(sphere, ellipsoid) to general domains.
As far as we know, this is the first rigorous study which captures a ground deformation description in such a generality.

\section{Integral representation and well-posedness of the direct problem} 

Since $\mathbb{C}:=\lambda \mathbf{I}\otimes \mathbf{I}+2\mu\mathbb{I}$,
the elastostatic Lam\'e operator $\mathcal{L}$ for a homogeneous and isotropic elastic medium
is given by
\begin{equation*}
	\mathcal{L}\boldsymbol{u}:=\textrm{div}(\mathbb{C}\widehat{\nabla}\boldsymbol{u})
		=\mu\Delta\boldsymbol{u}+(\lambda+\mu)\nabla\textrm{div}\,\boldsymbol{u},
\end{equation*}
where $\boldsymbol{u}$ represents the vector of the displacements and
$\widehat{\nabla}\bm{u}=\frac{1}{2}\bigl(\nabla\bm{u}+\nabla\bm{u}^T\bigr)$ the strain tensor.
With $\partial\bm{u}/\partial\bm{\nu}$ we depict the conormal derivative on the boundary of a domain,
that is the traction vector, which has the expression
\begin{equation*}
	\frac{\partial\bm{u}}{\partial\bm{\nu}}:=(\mathbb{C}\widehat{\nabla}\bm {u})\bm {n}
		=\lambda (\textrm{div}\,\bm{u})\bm{n}+2\mu(\widehat{\nabla}\bm{u})\bm{n}
\end{equation*}
or, equivalently,
\begin{equation*}
	\frac{\partial\bm{u}}{\partial\bm{\nu}}=2\mu\frac{\partial\bm{u}}{\partial\bm{n}}
		+\lambda(\textrm{div}\,\bm{u})\bm{n}+\mu(\bm{n}\times\textrm{rot}\,\bm{u}).
\end{equation*}
Here, we analyze the linear elastostatics boundary value problem 
\begin{equation}\label{direct problem}
\begin{cases}
\vspace{0.1cm}
\text{div}(\mathbb{C}\widehat{\nabla}\bm{u})=\bm{0} & \text{in}\ \mathbb{R}^3_-\setminus C\\
\vspace{0.1cm}
\displaystyle\frac{\partial \bm{u}}{\partial \bm{\nu}}=p\,\bm{n} & \text{on}\ 
\partial C\\
\vspace{0.1cm}
\displaystyle\frac{\partial\bm{u}}{\partial\bm{\nu}}=\bm{0} &  \text{on}\ \mathbb{R}^2\\
\bm{u}=o(\bm{1}),\quad \nabla \boldsymbol{u}=o\big(|\bm{x}|^{-1}\big) & |\bm{x}|\to\infty,
\end{cases} 
\end{equation}
where $C$ is the cavity and $p$ is a constant representing the pressure. For the Lam\'e parameters, we consider the physical
range $3\lambda+2\mu>0$ and $\mu>0$ which ensures positive definiteness of $\mathbb{C}$.
For the sequel, we recall that the positive definiteness of the tensor $\mathbb{C}$ implies the strongly ellipticity which
corresponds to the request $\mu>0$ and $\lambda+2\mu>0$, see \cite{Gurtin}.

The aim of this section is to provide an integral representation formula and to establish well-posedness of the problem.
To do that, we consider three steps
\begin{itemize}
\item firstly, we recall Betti's formulas, definition and some properties of single and double layer potentials of linear elasticity;
\item then, we give the expression of the fundamental solution $\mathbf{N}$ of the half-space with null traction on the boundary,
found by Mindlin in \cite{Mindlin36, Mindlin54};
\item finally, we represent the solution to \eqref{direct problem} by an integral formula through the fundamental solution of the half-space.
\end{itemize}
All these objects will be used to prove the well-posedness of the problem \eqref{direct problem}.

\subsection{Preliminaries} 
We recall Betti's formulas for the Lam\'e system which can be obtained by integration by parts, see for example
\cite{Ammari-libroelasticita,Kupradze}. Given a bounded Lipschitz domain $ C\subset\mathbb{R}^3$ and two
vectors $\bm{u}$,$\bm{v} \in\mathbb{R}^3$, the {\it first Betti formula} is
\begin{equation}\label{first Betti}
	\int\limits_{\partial C}\bm{u}\cdot \frac{\partial\bm{v}}{\partial \bm{\nu}}\,d\sigma(\bm{x})
	=\int\limits_{C}\bm{u}\cdot\mathcal{L}\bm{v}\, d\bm{x}+\int\limits_{C}Q(\bm{u},\bm{v})\, d\bm{x},
\end{equation}
where the quadratic form $Q$ associated to the Lam\'e system is
\begin{equation*}
	Q(\bm{u},\bm{v}):=\lambda (\textrm{div}\,\bm{u})(\textrm{div}\,\bm{v})+2\mu\widehat{\nabla}\bm{u}:\widehat{\nabla}\bm{v}.
\end{equation*}
From \eqref{first Betti} it is straightforward to find the {\it second Betti formula}
\begin{equation}\label{second Betti}
	\int\limits_{C}\left(\bm{u}\cdot\mathcal{L}\bm{v}-\bm{v}\cdot\mathcal{L}\bm{u}\right)\, d\bm{x}
	=\int\limits_{\partial C}\left(\bm{u}\cdot \frac{\partial\bm{v}}{\partial\bm{\nu}}
	-\bm{v}\cdot \frac{\partial\bm{u}}{\partial\bm{\nu}}\right)\,d\sigma(\bm{x}).
\end{equation}
Formula \eqref{first Betti} will be used to prove that the solution of \eqref{direct problem} is unique,
and the equality \eqref{second Betti} to get an integral representation formula for it.
To accomplish this second goal, a leading role is played by the fundamental solution of the Lam\'e system:
the {\it Kelvin matrix} $\bm{\Gamma}$ (or {\it Kelvin-Somigliana matrix}) solution to the equation
\begin{equation*}
	\textrm{div}(\mathbb{C}\widehat{\nabla}\bm{\Gamma})=\delta_{\bm{0}}\mathbf{I},\qquad
	\bm{x}\in\mathbb{R}^3\setminus\{\bm{0}\}, 
\end{equation*}
where $\delta_{\bm{0}}$ is the Dirac function centred at $\bm{0}$.
Setting $C_{\mu,\nu}:={1}/\{16\pi\mu(1-\nu)\}$, the explicit expression of $\mathbf{\Gamma}=(\Gamma_{ij})$ is
\begin{equation}\label{Gamma}
	\Gamma_{ij}(\bm{x})=-C_{\mu,\nu}\biggl\{\frac{(3-4\nu)\delta_{ij}}{|\bm{x}|}+\frac{x_i x_j}{|\bm{x}|^3}\biggr\},
	\qquad i,j=1,2,3,
\end{equation}
where $\delta_{ij}$ is the Kronecker symbol and $\Gamma_{ij}$ stands for the $i$-th component
of the displacement when a force is applied in the $j$-th direction at the point $\bm{0}$.
For reader convenience, we write also the gradient of $\mathbf{\Gamma}$ to highlight its behaviour at infinity
\begin{equation}\label{gamma gradient}
	\frac{\partial \Gamma_{ij}}{\partial x_k}(\bm{x})=C_{\mu,\nu}\biggl\{\frac{(3-4\nu)\delta_{ij} x_k
	-\delta_{ik}x_j-\delta_{jk}x_i}{|\bm{x}|^3}+\frac{3x_ix_jx_k}{|\bm{x}|^5}\biggr\},
	\qquad i,j,k=1,2,3.
\end{equation}
Therefore from \eqref{Gamma} and \eqref{gamma gradient} it is straightforward to see that 
\begin{equation}\label{gamma behaviour}
	|\mathbf{\Gamma}(\bm{x})|=O\left(\frac{1}{|\bm{x}|}\right)\quad\textrm{and}\quad
	|\nabla\mathbf{\Gamma}(\bm{x})|=O\left(\frac{1}{|\bm{x}|^2}\right)\qquad
	\textrm{as}\quad |\bm{x}|\to \infty.
\end{equation}
With the Kelvin matrix $\mathbf{\Gamma}$ at hand, we recall the definition of the single and double layer potentials
corresponding to the operator $\mathcal{L}$. 
Given $\bm{\varphi}\in \bm{L}^2(\partial\Omega)$ (see \cite{Ammari-libroelasticita,Ammari-Kang,Kupradze})
\begin{equation}\label{potentials}
	\begin{aligned}
	\mathbf{S}^{\Gamma}\bm{\varphi}(\bm{x})&:=\int\limits_{\partial C}\bm{\Gamma}(\bm{x}-\bm{y})\bm{\varphi}(\bm{y})\, d\sigma(\bm{y}),
	&\qquad &\bm{x}\in\mathbb{R}^3,\\
	\mathbf{D}^{\Gamma}\bm{\varphi}(\bm{x})&:=\int\limits_{\partial C}\frac{\partial\bm{\Gamma}}{\partial \bm{\nu}(\bm{y})}
		(\bm{x}-\bm{y})\bm{\varphi}(\bm{y})\, d\sigma(\bm{y}),
	&\qquad &\bm{x}\in\mathbb{R}^3\setminus\partial C,
	\end{aligned}
\end{equation}
where $\partial\mathbf{\Gamma}/\partial\bm{\nu}$ denotes the conormal derivative applied
to each column of the matrix $\mathbf{\Gamma}$.

In the sequel the subscripts $+$ and $-$ indicate the limits from outside and inside $C$, respectively. 
The double layer potential and the conormal derivative of the single layer potential satisfy the jump relations
\begin{equation}\label{jump relations}
	\begin{aligned}
	\mathbf{D}^{\Gamma}\bm{\varphi}\Big|_{\pm}(\bm{x})
		&=\left(\mp\tfrac{1}{2}\mathbf{I}+\mathbf{K}\right)\bm{\varphi}(\bm{x})
	&\quad &\textrm{for almost any}\, \bm{x}\in\partial C,\\
	\frac{\partial \mathbf{S}^{\Gamma}\bm{\varphi}}{\partial\bm{\nu}}\Big|_{\pm}(\bm{x})
		&=\left(\pm\tfrac{1}{2}\mathbf{I}+\mathbf{K}^*\right)\bm{\varphi}(\bm{x})
	&\quad &\textrm{for almost any}\, \bm{x}\in\partial C,\\
	\end{aligned}
\end{equation}
where $\mathbf{K}$ and $\mathbf{K}^*$ are the $\bm{L}^2$-adjoint Neumann-Poincar\'e boundary integral operators defined,
in the sense of Cauchy principal value, by
\begin{equation*}
	\begin{aligned}
	\mathbf{K}\bm{\varphi}(\bm{x})&:=\textrm{p.v.}\int\limits_{\partial C}
		\frac{\partial\bm{\Gamma}}{\partial\bm{\nu}(\bm{y})}(\bm{x}-\bm{y})\bm{\varphi}	(\bm{y})\, d\sigma(\bm{y}),\\
	\mathbf{K}^*\bm{\varphi}(\bm{x})&:=\textrm{p.v.}\int\limits_{\partial C}
		\frac{\partial\bm{\Gamma}}{\partial\bm{\nu}(\bm{x})}(\bm{x}-\bm{y})\bm{\varphi}(\bm{y})\, d\sigma(\bm{y}).
	\end{aligned}
\end{equation*}

It is worth noticing that these two operators are not compact even on smooth domains, in contrast with the analogous operators for the Laplace equation (see \cite{Ammari-Kang}), due to the presence in their kernels of the terms
\begin{equation*}
\frac{n_i(x_j-y_j)}{|\bm{x}-\bm{y}|^3}-\frac{n_j(x_i-y_i)}{|\bm{x}-\bm{y}|^3},\, \qquad i\neq j,
\end{equation*}
which make the kernel not integrable.  
Indeed, even in the case of smooth domains, we cannot approximate locally the terms $\bm{n}\times(\bm{x}-\bm{y})$ with
a smooth function, that is by power of $|\bm{x}-\bm{y}|$ via Taylor expansion, in order to obtain an integrable kernel on $\partial C$.
Therefore, the analysis to prove the invertibility of the operators in \eqref{jump relations} is intricate and usually based on regularizing operators (see \cite{Kupradze}) in the case of smooth domains. 
For the Lipschitz domains the analysis is much more involved and based on Rellich formulas (see \cite{DKV} and its companion article \cite{FKV}). 

\subsection{Fundamental solution of the half-space}
In this subsection we show the explicit expression of the Neumann function of the half-space presented for the first time
in \cite{Mindlin36} by means of Galerkin vector and nuclei of strain of the theory of linear elasticity, and secondly in \cite{Mindlin54}
using the Papkovich-Neuber representation of the displacement vector and the potential theory.

We consider the boundary value problem
\begin{equation}\label{pb mindlin}
\begin{cases}
\vspace{0.1cm}
\textrm{div}(\mathbb{C}\widehat{\nabla}\bm{v})=\bm{b} & \textrm{in}\, \mathbb{R}^3_-\\
\vspace{0.1cm}
\displaystyle\frac{\partial\bm{v}}{\partial\bm{\nu}}=\bm{0} & \textrm{on}\, \mathbb{R}^2\\
\vspace{0.1cm}
\bm{v}=o(\bm{1}),\quad \nabla\bm{v}=o(|\bm{x}|^{-1}) & \textrm{as}\, |\bm{x}|\to+\infty.
\end{cases}
\end{equation}
The Neumann function of \eqref{pb mindlin} is the kernel $\mathbf{N}$ of the integral operator
\begin{equation}\label{v}
	\bm{v}(\bm{x})=\int\limits_{\mathbb{R}^3_-}\mathbf{N}(\bm{x},\bm{y})\bm{b}(\bm{y})\, d\bm{y},
\end{equation}
giving the solution to the problem.

Given $\bm{y}=(y_1,y_2,y_3)$, we set $\widetilde{\bm{y}}:=(y_1,y_2,-y_3)$.

\begin{theorem}\label{thm:fundsol}
The Neumann function $\mathbf{N}$ of problem \eqref{pb mindlin} can be decomposed as
\begin{equation*}
	\mathbf{N}(\bm{x},\bm{y})=\mathbf{\Gamma}(\bm{x}-\bm{y})+\mathbf{R}^1(\bm{x}-\widetilde{\bm{y}})
		+y_3\mathbf{R}^2(\bm{x}-\widetilde{\bm{y}})+y_3^2\,\mathbf{R}^3(\bm{x}-\widetilde{\bm{y}}),
\end{equation*}
where $\mathbf{\Gamma}$ is the Kelvin matrix, see \eqref{Gamma},
and $\mathbf{R}^k$, $k=1,2,3$, have components $R^k_{ij}$ given by
\begin{equation*}
	\begin{aligned}
	R^1_{ij}(\bm{\eta})&:=C_{\mu,\nu}\bigl\{-(\tilde f+c_\nu\tilde g)\delta_{ij}-(3-4\nu)\eta_i\eta_j\tilde f^3\\
	&\hskip2.75cm +c_\nu\bigl[\delta_{i3}\eta_j-\delta_{j3}(1-\delta_{i3})\eta_i\bigr]\tilde f\tilde g
		+c_\nu(1-\delta_{i3})(1-\delta_{j3})\eta_i \eta_j\tilde f\tilde g^2\bigr\}\\
	R^2_{ij}(\bm{\eta})&:=2C_{\mu,\nu}\bigl\{(3-4\nu)\bigl[\delta_{i3}(1-\delta_{j3})\eta_j+\delta_{j3}(1-\delta_{i3})\eta_i\bigr]\tilde f^3
		-(1-2\delta_{3j})\delta_{ij}\eta_3\tilde f^3\\
	&\hskip8.75cm +3(1-2\delta_{3j})\eta_i\eta_j\eta_3\tilde f^5\bigr\}\\
	R^3_{ij}(\bm{\eta})&:=2C_{\mu,\nu}(1-2\delta_{j3})\bigl\{\delta_{ij} \tilde f^3-3\eta_i\eta_j\tilde f^5\bigr\}.
	\end{aligned}
\end{equation*}
for $i,j=1,2,3$, where $c_\nu:=4(1-\nu)(1-2\nu)$ and
\begin{equation*}
	\tilde f(\bm{\eta}):=\frac{1}{|\bm{\eta}|},\qquad
	\tilde g(\bm{\eta}):=\frac{1}{|\bm{\eta}|-\eta_3}.
\end{equation*}
\end{theorem}

\noindent
For the proof of Theorem \ref{thm:fundsol}, see the Appendix. Uniqueness of the solution to \eqref{pb mindlin} is similar to the one for problem \eqref{direct problem} which we
present in the following section.  
\vskip.25cm

The matrix $\mathbf{R}$, defined by
\begin{equation}\label{R}
	\mathbf{R}(\bm{\eta},y_3):=\mathbf{R}^1(\bm{\eta})+y_3\,\mathbf{R}^2(\bm{\eta})+y_3^2\,\mathbf{R}^3(\bm{\eta}),
\end{equation}
gives the regular part of the Neumann function since the singular point $\bm{\eta}=\bm{0}$
corresponds to $\bm{y}=(y_1,y_2,-y_3)$ with $y_3<0$, which belongs to $\mathbb{R}^3_+$.

To convert the problem \eqref{direct problem} into an integral form, bounds on the decay at infinity of the Neumann function and its derivative at infinity are needed.

\begin{proposition}
For any $M_x, M_y>0$, there exists $C>0$ such that
\begin{equation}\label{M behaviour}
	|\mathbf{N}(\bm{x},\bm{y})|\leq C\,|\bm{x}|^{-1}\quad\textrm{and}\quad
	|\nabla\mathbf{N}(\bm{x},\bm{y})|\leq C\,|\bm{x}|^{-2}
\end{equation}
for any $\bm{x},\bm{y}\in \mathbb{R}^3_-$ with $|\bm{x}|\geq M_x$ and $|\bm{y}|\leq M_y$.
\end{proposition}

\begin{proof}
If $\phi$ is a homogeneous function of degree $\alpha$ defined and continuous
in $\mathbb{R}^3_-\setminus\{\bm{0}\}$, then there exists a constant $C$ such that
\begin{equation*}
	|\phi(\bm{x})|\leq C|\bm{x}|^\alpha,\qquad \bm{x}\in \mathbb{R}^3_-\setminus\{\bm{0}\}.
\end{equation*}
Thus, since $\mathbf{R}^k$ are homogeneous of degree $-k$ for $k=1,2,3$ and
\begin{equation*}
	|\bm{\eta}|-\eta_3\geq |\bm{\eta}|=|\bm{x}-\widetilde{\bm{y}}|\geq |\bm{x}|-M_y
\end{equation*}
for $|\bm{x}|$ sufficiently large, the term $\mathbf{R}$ is bounded by
\begin{equation*}
	|\mathbf{R}|\leq |\mathbf{R}^1|+|y_3||\mathbf{R}^2|+|y_3|^2|\mathbf{R}^3|
		\leq C\left(\frac{1}{|\bm{x}|}+\frac{|y_3|}{|\bm{x}|^2}+\frac{|y_3|^2}{|\bm{x}|^3}\right)
		\leq \frac{C}{|\bm{x}|}.
\end{equation*}
Coupling with \eqref{gamma behaviour}, we deduce the bound for $\mathbf{N}$.

The estimates on $|\nabla\mathbf{N}|$ is consequence of the homogeneity of derivatives of homogeneous
functions together with the observation that $\tilde f$ and $\tilde g$ are $C^1$ in $\mathbb{R}^3_-\setminus\{\bm{0}\}$.
\end{proof}
 
\subsection{Representation formula}
Next, we derive an integral representation formula for $\bm{u}$ solution to problem \eqref{direct problem}.
For, we make use of single and double layer potentials defined in \eqref{potentials}
and integral contributions relative to the regular part $\mathbf{R}$ of the Neumann function $\mathbf{N}$,
defined by
\begin{equation}\label{regular potentials}
	\begin{aligned}
	\mathbf{S}^R\bm{\varphi}(\bm{x})&:=\int\limits_{\partial C}(\mathbf{R}(\bm{x},\bm{y}))^T\bm{\varphi}(\bm{y})\, d\sigma(\bm{y}),
		&\qquad &\bm{x}\in\mathbb{R}^3_-,\\
	\mathbf{D}^R\bm{\varphi}(\bm{x})&:=\int\limits_{\partial C}\left(\frac{\partial\mathbf{R}}{\partial \bm{\nu}(\bm{y})}(\bm{x},\bm{y})\right)^T
		\bm{\varphi}(\bm{y})\, d\sigma(\bm{y}),
		&\qquad &\bm{x}\in\mathbb{R}^3_-,
	\end{aligned}
\end{equation}
where $\bm{\varphi}\in \bm{L}^2(\partial C)$.

\begin{theorem}\label{th of repr formula}
The solution $\bm{u}$ to \eqref{direct problem} is such that
\begin{equation}\label{representation formula}
\bm{u}=p\mathbf{S}^{\Gamma}\bm{n}-\mathbf{D}^{\Gamma}\bm{f}+p\mathbf{S}^R\bm{n}-\mathbf{D}^R\bm{f},\qquad \textrm{in}\,\  \mathbb{R}^3_-\setminus \overline{C}
\end{equation}
where $\mathbf{S}^{\Gamma}$, $\mathbf{D}^{\Gamma}$ are defined in \eqref{potentials}, $\mathbf{S}^R$, $\mathbf{D}^R$ in \eqref{regular potentials}, $p\bm{n}$ is the boundary condition in \eqref{direct problem} and $\bm{f}$ is the trace of $\bm{u}$ on $\partial C$.
\end{theorem} 
Before proving this theorem, we observe that $\bm{f}$ solves the integral equation
\begin{equation}\label{trace equation}
	\left(\tfrac{1}{2}\mathbf{I}+\mathbf{K}+\mathbf{D}^R\right)\bm{f}
	=p\left(\mathbf{S}^{\Gamma}\bm{n}+\mathbf{S}^{R}\bm{n}\right),\qquad \textrm{on}\, \ \partial C,
\end{equation}
obtained by the application of the trace properties of the double layer potential \eqref{jump relations} in formula \eqref{representation formula}.

\begin{proof}[Proof of Theorem \ref{th of repr formula}]
Given $r, \varepsilon>0$ such that $C\subset B_r(\bm{0})$ and $B_{\varepsilon}(\bm{y})\subset\mathbb{R}^3_-\setminus \overline{C}$, let
\begin{equation*}
	\Omega_{r,\varepsilon}=\left(\mathbb{R}^3_-\cap B_{r}(\bm{0})\right)\setminus \left(C\cup B_{\varepsilon}(\bm{y})\right)
\end{equation*}  
with $r$ sufficiently large such that to contain the cavity $C$; additionally, we define $\partial B^h_{r}(\bm{0})$ as the intersection of the hemisphere with the boundary of the half-space, and with $\partial B^b_r(\bm{0})$ the spherical cap (see Figure \ref{figure hemisphere with cavity}).
\begin{figure}[h]
\centering
\includegraphics[scale=0.45]{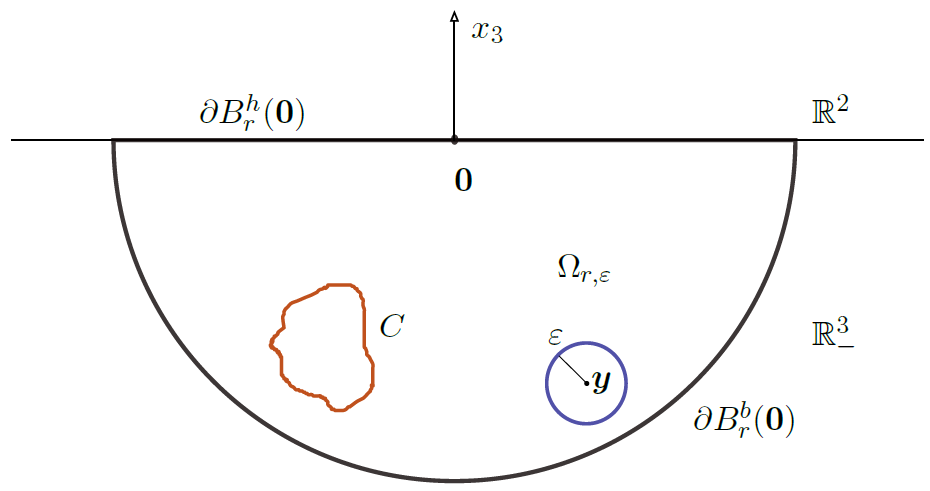}
\caption{Domain $\varOmega_{r,\varepsilon}$.}\label{figure hemisphere with cavity}
\end{figure}
Now, we apply Betti's formula \eqref{second Betti} to $\bm{u}$ and the $k$-th column vector of $\mathbf{N}$, indicated by $\bm{N}^{(k)}$, for $k=1,2,3$, in $\Omega_{r,\varepsilon}$, hence 
\begin{equation*}
\begin{aligned}
	0&=\int\limits_{\Omega_{r,\varepsilon}}\left[\bm{u}(\bm{x})\cdot \mathcal{L} \bm{N}^{(k)}(\bm{x},\bm{y})-\bm{N}^{(k)}(\bm{x},\bm{y})\cdot \mathcal{L}\bm{u}(\bm{x})\right]\,d\bm{x}\\
	&=\int\limits_{\partial B_r^b(\bm{0})}\left[\frac{\partial\bm{N}^{(k)}}{\partial \bm{\nu_x}}(\bm{x},\bm{y})\cdot \bm{u}(\bm{x})-\bm{N}^{(k)}(\bm{x},\bm{y})\cdot \frac{\partial\bm{u}}{\partial\bm{\nu_x}}(\bm{x})\right]\, d\sigma(\bm{x})\\
	&\quad -\int\limits_{\partial B_{\varepsilon}(\bm{y})}\left[\frac{\partial\bm{N}^{(k)}}{\partial \bm{\nu_x}}(\bm{x},\bm{y})\cdot \bm{u}(\bm{x})-\bm{N}^{(k)}(\bm{x},\bm{y})\cdot \frac{\partial\bm{u}}{\partial\bm{\nu_x}}(\bm{x})\right]\, d\sigma(\bm{x})\\
	&\quad -\int\limits_{\partial C}\left[\frac{\partial\bm{N}^{(k)}}{\partial \bm{\nu_x}}(\bm{x},\bm{y})\cdot \bm{u}(\bm{x})-\bm{N}^{(k)}(\bm{x},\bm{y})\cdot \frac{\partial\bm{u}}{\partial\bm{\nu_x}}(\bm{x})\right]\, d\sigma(\bm{x})\\
	&:=I_1+I_2+I_3,
\end{aligned}
\end{equation*}
since, from \eqref{direct problem} and the boundary condition in \eqref{pb mindlin},
\begin{equation*}
	\int\limits_{\partial B_r^h(\bm{0})}\left[\frac{\partial\bm{N}^{(k)}}{\partial \bm{\nu_x}}(\bm{x},\bm{y})\cdot \bm{u}(\bm{x})
	-\bm{N}^{(k)}(\bm{x},\bm{y})\cdot \frac{\partial\bm{u}}{\partial\bm{\nu_x}}(\bm{x})\right]\, d\sigma(\bm{x})=0.
\end{equation*}
We show that the term $I_1$ goes to zero by using the behaviour at infinity of $\bm{u}$ given in \eqref{direct problem} and of the Neumann function given in \eqref{M behaviour}. Indeed, we have 
\begin{equation*}
	\Biggl|\;\int\limits_{\partial B_r^b(\bm{0})}\frac{\partial\bm{N}^{(k)}}{\partial \bm{\nu_x}}(\bm{x},\bm{y})\cdot \bm{u}(\bm{x})
		\,d\sigma(\bm{x})\Biggr|
	\leq \int\limits_{\partial B_r^b(\bm{0})}|\bm{u}| \Bigg|\frac{\partial\bm{N}^{(k)}}{\partial\bm{\nu_x}}\Bigg|\, d\sigma(\bm{x})
	\leq \frac{C}{r^2}\int\limits_{\partial B_r^b(\bm{0})}|\bm{u}(\bm{x})|d\sigma(\bm{x}).
\end{equation*}
This last integral can be estimated by means of the spherical coordinates 
$x_1=r \sin\varphi\cos\theta$, $x_2=r \sin\varphi\sin\theta$, $x_3=r \cos\varphi$
where $\varphi\in [\pi/2,\pi]$, since $B_r^b(\bm{0})$ is a hemisphere in $\mathbb{R}^3_-$, and $\theta\in[0,2\pi)$, indeed
\begin{equation*}
	\begin{aligned}
	\frac{C}{r^2}\int\limits_{\partial B^b_r(\bm{0})}|\bm{u}|\, d\sigma(\bm{x})
	&=C\int\limits_{\frac{\pi}{2}}^{\pi}\int\limits_{0}^{2\pi}|\bm{u}(r,\theta,\varphi)|\sin\varphi\,d\theta\, d\varphi\\
	&\leq C\sup\limits_{\theta\in[0,2\pi),\varphi\in[\frac{\pi}{2},\pi]}|\bm{u}(r,\theta,\varphi)|\to 0,
	\end{aligned}
\end{equation*}
as $r\to +\infty$, since $\bm{u}=o(\bm{1})$. Similarly
\begin{equation*}
	\Biggl|\;\int\limits_{\partial B_r^b(\bm{0})}\bm{N}^{(k)}(\bm{x},\bm{y})\cdot \frac{\partial\bm{u}}{\partial\bm{\nu_x}}(\bm{x})\,
		d\sigma(\bm{x})\Biggr|
		\leq\int\limits_{\partial B^b_r(\bm{0})}|\bm{N}^{(k)}|\, \Big|\frac{\partial \bm{u}}{\partial\bm{\nu_x}}\Big|\, d\sigma(\bm{x})
	\leq \frac{C}{r}\int\limits_{\partial B^b_r}\Big|\frac{\partial\bm{u}}{\partial\bm{\nu_x}}\Big|\, d\sigma(\bm{x}).
\end{equation*}
Again, passing through spherical coordinates, we get
\begin{equation}\label{int on Bbr}
	\frac{C}{r}\int\limits_{\partial B^b_r(\bm{0})}\Big|\frac{\partial\bm{u}}{\partial\bm{\nu_x}}\Big|\, d\sigma(\bm{x})
	\leq C\sup\limits_{\theta\in[0,2\pi),\varphi\in[\frac{\pi}{2},\pi]}r\Big|\frac{\partial\bm{u}}{\partial\bm{\nu}}(r,\theta,\varphi)\Big|\to 0,
\end{equation}
as $r\to+\infty$, since $|\nabla\bm{u}|=o(r^{-1})$.\\
Integral $I_2$ gives the value of the function $\bm{u}$ in $\bm{y}$ as $\varepsilon$ goes to zero. Indeed, we have
\begin{equation*}
	\begin{aligned}
	&\Biggl|\;\int\limits_{\partial B_{\varepsilon}(\bm{y})} \bm{N}^{(k)}(\bm{x},\bm{y})\cdot \frac{\partial\bm{u}}{\partial\bm{\nu_x}}(\bm{x})
		\, d\sigma(\bm{x})\Biggr|
		\leq \int\limits_{\partial B_{\varepsilon}(\bm{y})}|\bm{N}^{(k)}|\Big|\frac{\partial\bm{u}}{\partial\bm{\nu_x}}\Big|\, d\sigma(\bm{x})\\
	&\hskip4.5cm
		\leq \sup\limits_{\bm{x}\in\partial B_{\varepsilon}(\bm{y})}\Bigl|\frac{\partial\bm{u}}{\partial\bm{\nu_x}}\Bigr|
		\int\limits_{\partial B_{\varepsilon}(\bm{y})}\left[|\bm{\Gamma}^{(k)}|+|\bm{R}^{(k)}|\right]\, d\sigma(\bm{x})
		=O(\varepsilon),
	\end{aligned}
\end{equation*}
since the second integral has a continuous kernel.
On the other hand
\begin{equation*}
	\begin{aligned}
	&-\int\limits_{\partial B_{\varepsilon}(\bm{y})}\frac{\partial\bm{N}^{(k)}}{\partial\bm{\nu_x}}(\bm{x},\bm{y})
		\cdot \bm{u}(\bm{x})\, d\sigma(\bm{x})\\
	&\hskip1cm
		=-\bm{u}(\bm{y})\cdot\int\limits_{\partial B_{\varepsilon}(\bm{y})}
		\frac{\partial\bm{N}^{(k)}}{\partial\bm{\nu_x}}(\bm{x},\bm{y})\, d\sigma(\bm{x})
		+\int\limits_{\partial B_{\varepsilon}(\bm{y})}[\bm{u}(\bm{y})-\bm{u}(\bm{x})]
		\cdot \frac{\partial\bm{N}^{(k)}}{\partial\bm{\nu_x}}(\bm{x},\bm{y})\, d\sigma(\bm{x}).
	\end{aligned}
\end{equation*}
The latter integral tends to zero when $\varepsilon$ goes to zero because
\begin{equation*}
	\Bigg|\int\limits_{\partial B_{\varepsilon}(\bm{y})}[\bm{u}(\bm{y})-\bm{u}(\bm{x})]\cdot \frac{\partial\bm{N}^{(k)}}{\partial\bm{\nu_x}}(\bm{x},\bm{y})\, d\sigma(\bm{x})\Bigg|\leq \sup\limits_{\bm{x}\in\partial B_{\varepsilon}(\bm{y})}|\bm{u}(\bm{y})-\bm{u}(\bm{x})|\int\limits_{\partial B_{\varepsilon}(\bm{y})}\Bigg|\frac{\partial\bm{N}^{(k)}}{\partial\bm{\nu_x}}\Bigg|\, d\sigma(\bm{x})
\end{equation*} 
and this last integral is bounded when $\varepsilon$ goes to zero. 
Let finally observe that
\begin{equation}\label{integral I31 first}
\begin{aligned}
	-\bm{u}(\bm{y})\cdot&\int\limits_{\partial B_{\varepsilon}(\bm{y})}\frac{\partial\bm{N}^{(k)}}{\partial\bm{\nu_x}}(\bm{x},\bm{y})\, d\sigma(\bm{x})=-\bm{u}(\bm{y})\cdot\int\limits_{\partial B_{\varepsilon}(\bm{y})}\frac{\partial(\bm{\Gamma}^{(k)}+\bm{R}^{(k)})}{\partial\bm{\nu_x}}(\bm{x},\bm{y})\, d\sigma(\bm{x})\\
	&=-\bm{u}(\bm{y})\cdot\int\limits_{\partial B_{\varepsilon}(\bm{y})}\frac{\partial\bm{\Gamma}^{(k)}}{\partial\bm{\nu_x}}(\bm{x}-\bm{y})\, d\sigma(\bm{x})-\bm{u}(\bm{y})\cdot\int\limits_{\partial B_{\varepsilon}(\bm{y})}\frac{\partial\bm{R}^{(k)}}{\partial\bm{\nu_x}}(\bm{x},\bm{y})\, d\sigma(\bm{x}),
\end{aligned}
\end{equation}
where the latter integral tends to zero as $\varepsilon\to 0$, since $\bm{R}^{(k)}$ represents the regular part of the Neumann function.
To deal with the first integral, we preliminarly observe that direct differentiation gives\\
\begin{equation}\label{traction Gamma}
	\begin{aligned}
	\left(\frac{\partial\bm{\Gamma}^{(k)}}{\partial\bm{\nu_x}}\right)_h(\bm{x}-\bm{y})
	&=-c'_\nu\Bigg\{n_k(\bm{x})\frac{\partial}{\partial x_h}\frac{1}{|\bm{x}-\bm{y}|}	
			-n_h(\bm{x})\frac{\partial}{\partial x_k}\frac{1}{|\bm{x}-\bm{y}|}\\
	&\quad+\Bigg[\delta_{hk}+\frac{3}{(1-2\nu)}\frac{\partial|\bm{x}-\bm{y}|}{\partial x_k}\frac{\partial|\bm{x}-\bm{y}|}{\partial x_h}\Bigg]\frac{\partial}{\partial \bm{n}(\bm{x})}\frac{1}{|\bm{x}-\bm{y}|}\Bigg\},
	\end{aligned}
\end{equation}
where $c'_\nu:=(1-2\nu)/(8\pi(1-\nu))$.\\
We substitute this expression into the integral \eqref{integral I31 first} and we take into account that
\begin{equation*}
n_h(\bm{x})=\frac{x_h-y_h}{|\bm{x}-\bm{y}|},\qquad \frac{\partial}{\partial x_k}\frac{1}{|\bm{x}-\bm{y}|}=-\frac{x_k-y_k}{|\bm{x}-\bm{y}|^3},
\end{equation*}
hence
\begin{equation*}
\int\limits_{\partial B_{\varepsilon}(\bm{y})}n_h(\bm{x})\frac{\partial}{\partial x_k}\frac{1}{|\bm{x}-\bm{y}|}\, d\sigma(\bm{x})=-\int\limits_{\partial B_{\varepsilon}(\bm{y})}\frac{(x_h-y_h)(x_k-y_k)}{|\bm{x}-\bm{y}|^4}\, d\sigma(\bm{x}).
\end{equation*}
To solve this last integral we use spherical coordinates, that is
\begin{equation*}
x_1-y_1=\varepsilon \sin\varphi\cos\theta,\qquad
x_2-y_2=\varepsilon \sin\varphi\sin\theta,\qquad
x_3-y_3=\varepsilon\cos\varphi,
\end{equation*}
where $\varphi\in [0,\pi]$ and $\theta\in[0,2\pi)$. 
From a simple calculation it follows
\begin{equation}\label{first integral spherical coordinate}
-\int\limits_{\partial B_{\varepsilon}(\bm{y})}\frac{(x_h-y_h)(x_k-y_k)}{|\bm{x}-\bm{y}|^4}\, d\sigma(\bm{x})=
\begin{cases}
0 & \textrm{if}\, h\neq k\\
-\frac{4}{3}\pi& \textrm{if}\, h=k.
\end{cases}
\end{equation}
Therefore, from \eqref{traction Gamma} and \eqref{first integral spherical coordinate}, we have
\begin{equation}\label{I31 null integral}
	\int\limits_{\partial B_{\varepsilon}(\bm{y})}\left(n_k(\bm{x})\frac{\partial}{\partial x_h}\frac{1}{|\bm{x}-\bm{y}|}
		-n_h(\bm{x})\frac{\partial}{\partial x_k}\frac{1}{|\bm{x}-\bm{y}|}\right)\, d\sigma(\bm{x})=0,
\end{equation}
for any $h$ and $k$.
Hence, \eqref{integral I31 first} becomes   
\begin{equation*}
	\begin{aligned}
	&-\bm{u}(\bm{y})\cdot\int\limits_{\partial B_{\varepsilon}(\bm{y})}\frac{\partial\bm{N}^{(k)}}{\partial\bm{\nu_x}}(\bm{x}-\bm{y})\, d\sigma(\bm{x})\\
	&=c'_\nu\sum\limits_{h=1}^{3}u_h(\bm{y})\int\limits_{\partial B_{\varepsilon}(\bm{y})}\left[\left(\delta_{hk}
	+\frac{3}{(1-2\nu)}\frac{\partial|\bm{x}-\bm{y}|}{\partial x_k}\frac{\partial|\bm{x}-\bm{y}|}{\partial x_h}\right)
	\frac{\partial}{\partial \bm{n}_{\bm{x}}}\frac{1}{|\bm{x}-\bm{y}|}\right]\, d\sigma(\bm{x})+O(\varepsilon).
	\end{aligned}
\end{equation*}
Employing again the spherical coordinates and the definition of $c'_\nu$, we find that
\begin{equation}\label{I31A}
	\frac{1-2\nu}{8\pi(1-\nu)}\int\limits_{\partial B_{\varepsilon}(\bm{y})}\delta_{hk}\frac{\partial}{\partial \bm{n}_{\bm{x}}}\frac{1}{|\bm{x}-\bm{y}|}\, d\sigma(\bm{x})=
	\begin{cases}
	-\frac{1-2\nu}{2(1-\nu)} & \textrm{if}\, h=k\\
	0 & \textrm{if}\, h\neq k.
	\end{cases} 
\end{equation}
Similarly
\begin{equation}\label{I31B}
	\frac{3}{8\pi(1-\nu)}\int\limits_{\partial B_{\varepsilon}(\bm{y})}\left(\frac{\partial |\bm{x}-\bm{y}|}{\partial x_k}\frac{\partial |\bm{x}-\bm{y}|}{\partial x_h}\right)\frac{\partial}{\partial \bm{n}_{\bm{x}}}\frac{1}{|\bm{x}-\bm{y}|}\, d\sigma(\bm{x})=
	\begin{cases}
	-\frac{1}{2(1-\nu)} & \textrm{if}\, h=k\\
	0 & \textrm{if}\, h\neq k.
\end{cases} 
\end{equation}
Putting together all the results in \eqref{I31A} and \eqref{I31B}, we find that
\begin{equation*}
\begin{aligned}
\lim\limits_{\varepsilon\to 0}\left(-\bm{u}(\bm{y})\cdot\int\limits_{\partial B_{\varepsilon}(\bm{y})}\frac{\partial\bm{N}^{(k)}}{\partial\bm{\nu_x}}(\bm{x}-\bm{y})\, d\sigma(\bm{x})\right)=-u_k(\bm{x}).
\end{aligned}
\end{equation*} 
Using the definition of single and double layer potentials \eqref{potentials}, \eqref{regular potentials} and splitting $\mathbf{N}$ as $\mathbf{\Gamma}+\mathbf{R}$ formula \eqref{representation formula} holds.
\end{proof}

From the behaviour of the Neumann function given in \eqref{M behaviour} and the representation formula
in \eqref{representation formula}, we immediately get

\begin{corollary}\label{corollary}
If $\bm{u}$ is a solution to \eqref{direct problem}, then
\begin{equation}\label{u as O}
	\bm{u}(\bm{y})=O(|\bm{y}|^{-1})\qquad\textrm{as}\quad |\bm{y}|\to \infty.
\end{equation}
\end{corollary}

\subsection{Well-posedness}
The well-posedness of the boundary value problem \eqref{direct problem} reduces to show the invertibility of 
\begin{equation}\label{trace operator}
	\tfrac{1}{2}\mathbf{I}+\mathbf{K}+\mathbf{D}^R:\bm{L}^2(\partial C)\to\bm{L}^2(\partial C). 
\end{equation}
In particular, in order to prove the injectivity of the operator \eqref{trace operator} we show the uniqueness of $\bm{u}$ following the classical approach based on the application of the
Betti's formula \eqref{first Betti} and the energy method, see \cite{Fichera,Kupradze}.
From the injectivity, it follows the existence of $\bm{u}$ proving the surjectivity of \eqref{trace operator} which is obtained by the application of the index theory regarding bounded and linear operators.

First of all, let us recall the closed range theorem due to Banach (see \cite{Kato,Yosida}).
\begin{theorem}\label{Banach theorem}
Let $X$ and $Y$ be Banach spaces, and $T$ a closed linear operator defined in $X$ into $Y$ such that $\overline{D(T)}=X$. Then the following propositions are all equivalent:\par
a.\ \ $\textrm{Im}(T)$ is closed in $Y$;\qquad
b.\ \ $\textrm{Im}(T^*)$ is closed in $X^*$;\par
c.\ \ $\textrm{Im}(T)=(\textrm{Ker}(T^*))^{\perp}$;\qquad
d.\ \ $\textrm{Im}(T^*)=(\textrm{Ker}(T))^{\perp}$.  
\end{theorem}
Through this theorem we can prove

\begin{lemma}\label{lemma 1/2I+K}
The operator  $\tfrac{1}{2}\mathbf{I}+\mathbf{K}:\bm{L}^2(\partial C)\to \bm{L}^2(\partial C)$ is invertible with bounded inverse.
\end{lemma}

\begin{proof}
The assertion of this lemma is based on the invertibility of the operator $\frac12\mathbf{I}+\mathbf{K}^*$ studied in \cite{DKV};
it is known that
\begin{equation*}
	\tfrac{1}{2}\mathbf{I}+\mathbf{K}^*:\bm{L}^2(\partial C)\to \bm{L}^2(\partial C)
\end{equation*}
is a bounded linear operator, injective and with dense and closed range.  
Therefore, from Theorem \ref{Banach theorem} we have
\begin{equation*}
	\textrm{Ker}\left(\tfrac{1}{2}\mathbf{I}+\mathbf{K}\right)=\{0\},\qquad\qquad
	\textrm{Im}\left(\tfrac{1}{2}\mathbf{I}+\mathbf{K}\right)^{\perp}=\{0\} 
\end{equation*}
and $\textrm{Im}(1/2\mathbf{I}+\mathbf{K})$ is closed.
Then, it follows that the operator $\tfrac{1}{2}\mathbf{I}+\mathbf{K}:\bm{L}^2(\partial C)\to\bm{L}^2(\partial C)$ is bijective
and the assertion follows exploiting the bounded inverse theorem.  
\end{proof}

Since $\mathbf{D}^R$ has a continuous kernel we prove its compactness adapting the arguments contained in \cite{Kress}.

\begin{lemma}\label{lemma D}
The operator  $\mathbf{D}^R:\bm{L}^2(\partial C)\to\bm{L}^2(\partial C)$ is compact.
\end{lemma}

\begin{proof}
For the sake of simplicity, we call
\begin{equation*}
	\mathbf{H}(\bm{x},\bm{y}):=\frac{\partial\mathbf{R}}{\partial\bm{\nu}}(\bm{x},\bm{y}),\qquad \bm{x},\bm{y}\in\partial C
\end{equation*}
and we denote by $\bm{H}^{(k)}$,  $k=1,2,3$, the column vectors of the matrix $\mathbf{H}$.

Let $\bm{S}$ be a bounded set such that $\bm{S}\subset \bm{L}^2(\partial C)$, that is $\|\bm{\varphi}\|_{\bm{L}^2(\partial C)}\leq K$,
for any $\bm{\varphi}\in \bm{S}$.
Then, applying Cauchy-Schwarz inequality
\begin{equation*}
	|(\bm{D}^R\bm{\varphi}(\bm{y}))_k|^2
	\leq \|\bm{H}^{(k)}(\cdot,\bm{y})\|_{\bm{L}^2(\partial C)}^2\|\bm{\varphi}\|_{\bm{L}^2(\partial C)}^2
	\leq K |\partial C| \max\limits_{\bm{x},\bm{y}\in\partial C}|\bm{H}^{(k)}|,
\end{equation*}   
with $k=1,2,3$, for all $\bm{y}\in\partial C$ and $\bm{\varphi}\in \bm{S}$.
Hence $|\bm{D}^R(\bm{\varphi})|\leq K'$, with $ K'>0$, which implies that $\bm{D}^R(\bm{S})$ is bounded.
Moreover, for all $\varepsilon>0$ there exist $\bm{\varphi},\bm{\varphi}'\in \bm{S}$ and $\delta>0$ such that
if $\|\bm{\varphi}(\bm{y})-\bm{\varphi'}(\bm{y})\|_{\bm{L}^2(\partial C)}<\delta$ then, applying again the Cauchy-Schwarz inequality                                                                                                                                                                                                            
\begin{equation*}
	|\bm{D}^R(\bm{\varphi}-\bm{\varphi}')(\bm{y})|< \varepsilon.
\end{equation*}
Thus $\mathbf{D}^R(\bm{S})\subset \mathbf{C}(\partial C)$ where $\mathbf{C}(\partial C)$ indicates the space of continuous function on $\partial C$.
Since each component of the matrix $\mathbf{H}$ is uniformly continuous on the compact set $\partial C \times \partial C$,
for every $\varepsilon>0$ there exists $\delta>0$ such that
\begin{equation*}
	|\bm{H}^{(k)}(\bm{z},\bm{x})-\bm{H}^{(k)}(\bm{z},\bm{y})|\leq \frac{\varepsilon}{\sqrt{3}K|\partial C|^{1/2}},
\end{equation*}
for all $\bm{x},\bm{y},\bm{z}\in\partial C$ with $|\bm{x}-\bm{y}|<\delta$.
Since
\begin{equation*}
	\begin{aligned}
	|(\bm{D}^R\bm{\varphi})_k(\bm{x})-(\bm{D}^R\bm{\varphi})_k(\bm{y})|
	&\leq \int\limits_{\partial C}|\bm{H}^{(k)}(\bm{z},\bm{x})-\bm{H}^{(k)}(\bm{z},\bm{y})| |\bm{\varphi}(\bm{z})|\, d\sigma(\bm{z})\\
	&\leq \|\bm{H}^{(k)}(\cdot,\bm{x})-\bm{H}^{(k)}(\cdot,\bm{y})\|_{\bm{L}^2(\partial C)}\|\bm{\varphi}\|_{\bm{L}^2(\partial C)}
		\leq \frac{\varepsilon}{\sqrt{3}},
	\end{aligned}
\end{equation*}
for $k=1,2,3$, hence
\begin{equation*}
	|(\bm{D}^R\bm{\varphi})(\bm{x})-(\bm{D}^R\bm{\varphi})(\bm{y})|\leq \varepsilon,
\end{equation*}
for all $\bm{x},\bm{y}\in\partial C$ and $\bm{\varphi}\in\bm{S}$, that is $\mathbf{D}^R(\bm{S})$ is equicontinuous.
The assertion follows from Ascoli-Arzel\`a Theorem and noticing that $\bm{C}(\partial C)$ is dense in $\bm{L}^2(\partial C)$.
\end{proof}

We now prove

\begin{theorem}[uniqueness]\label{th uniqueness op}
The boundary valure problem \eqref{direct problem} admits a unique solution.
\end{theorem}

\begin{proof}
Let $\bm{u}^1$ and $\bm{u}^2$ be solutions to \eqref{direct problem}.
Then the difference $\bm{v}:=\bm{u}^1-\bm{u}^2$ solves the homogeneous version of $\eqref{direct problem}$, that is 
\begin{equation}\label{elasticity equation}
	\text{div}(\mathbb{C}\widehat{\nabla}\bm{v})=\bm{0}\,\qquad  \text{in}\ \mathbb{R}^3_-\setminus C
\end{equation}
with homogeneous boundary conditions
\begin{equation}\label{bc hom problem}
	\frac{\partial \bm{v}}{\partial \bm{\nu}}=\bm{0}\,\qquad  \text{on}\ \partial C,\qquad 
	\frac{\partial\bm{v}}{\partial\bm{\nu}}=\bm{0}\,\qquad   \text{on}\ \mathbb{R}^2\\ 
\end{equation}
\begin{equation*}
	\bm{v}=O(|\bm{x}|^{-1})\,\qquad  \nabla \boldsymbol{v}=o\big(|\bm{x}|^{-1}\big)\,\qquad  |\bm{x}|\to\infty,
\end{equation*}
where we make use of the decay condition at infinity comes from Corollary \ref{corollary}.\\
Applying Betti's formula \eqref{first Betti} to $\bm{v}$ in $\Omega_r=(\mathbb{R}^3_-\cap B_r(\bm{0}))\setminus C$, we find
\begin{equation*}
	\int\limits_{\partial \Omega_r}\bm{v}\cdot \frac{\partial\bm{v}}{\partial\bm{\nu}}\, d\sigma(\bm{x})
	=\int\limits_{\Omega_r}Q(\bm{v},\bm{v})\, d\bm{x}
\end{equation*} 
where $Q$ is the quadratic form $Q(\bm{v},\bm{v})=\lambda(\textrm{div}\,\bm{v})^2+2\mu|\widehat{\nabla}\bm{v}|^2$.
From the behaviour of $\bm{v}$ and the boundary conditions \eqref{bc hom problem}, we estimate the previous integral
defined on the surface of $\Omega_r$; contributions over the surface of the cavity and the intersection of the hemisphere with the
half-space are null by means of \eqref{bc hom problem}, whereas on the spherical cap
\begin{equation*}
\Bigg|\int\limits_{\partial B^b_r(\bm{0})}\bm{v}\cdot\frac{\partial\bm{v}}{\partial\bm{\nu}}\,d\sigma(\bm{x})\Bigg|\leq\int\limits_{\partial B^b_r(\bm{0})} |\bm{v}|\Big|\frac{\partial\bm{v}}{\partial\bm{\nu}}\Big|\,d\sigma(\bm{x})\leq \frac{C}{r}\int\limits_{\partial B^b_r(\bm{0})}\Big|\frac{\partial\bm{v}}{\partial\bm{\nu}}\Big|\,d\sigma(\bm{x}).
\end{equation*}
As already done in \eqref{int on Bbr} to obtain the representation formula, this integral can be evaluated by spherical coordinates;
in particular, it tends to zero when $r\to+\infty$.
Therefore
\begin{equation*}
	\int\limits_{\mathbb{R}^3_-\setminus C} \left\{\lambda(\textrm{div}\,\bm{v})^2+2\mu|\widehat{\nabla}\bm{v}|^2\right\}d\bm{x}=0.
\end{equation*}
Since the quadratic form is positive definite for the parameters range $3\lambda+2\mu>0$ and $\mu>0$, we have that
\begin{equation}\label{rigid disp}
	\widehat{\nabla}\bm{v}=\mathbf{0}\quad \textrm{and}\quad \textrm{div}\,\bm{v}=0\quad \textrm{in}\,\ \mathbb{R}^3_-\setminus C, 
\end{equation}
It follows that the rigid displacements $\bm{v}=\bm{a}+\mathbf{A}\bm{x}$, with $\bm{a}\in\mathbb{R}^3$ and $\mathbf{A}$ belonging
to the space of the anti-symmetric matrices (see \cite{Ammari-libroelasticita,Ciar88}), could be the only nonzero solutions which satisfy 
\eqref{elasticity equation}, the boundary conditions in \eqref{bc hom problem} and \eqref{rigid disp}.
However, in this case they are excluded thanks to the behaviour of the function $\bm{v}$ at infinity.
Hence, we obtain that $\bm{v}=\bm{0}$, that is $\bm{u}^1=\bm{u}^2$ in $\mathbb{R}^3_-\setminus C$.
\end{proof}

The uniqueness result for problem \eqref{direct problem} ensures the injectivity of the operator \eqref{trace operator}. 
In order to prove the surjectivity of the operator \eqref{trace operator} we recall, for the reader convenience, the definition of the index of an operator (see \cite{Abramovich-Aliprantis,Kato})
\begin{definition}
Given a bounded operator $T:X\to Y$ between two Banach spaces, the index of the operator $T$ is the extended real number defined as
\begin{equation*}
i(T)=\textrm{dim}(\textrm{Ker}(T))-\textrm{dim}(Y/\textrm{Im}(T)),
\end{equation*}
where $\textrm{dim}(\textrm{Ker}(T))$ is called the \textbf{nullity} and $\textrm{dim}(Y/\textrm{Im}(T))$  the \textbf{defect} of $T$. In particular, when the nullity and the defect are both finite the operator $T$ is said to be \textbf{Fredholm}.
\end{definition}
We remember also an important theorem regarding the index of a bounded linear operator perturbed with a compact operator (see \cite{Abramovich-Aliprantis}).
\begin{theorem}\label{index of a compact pertubation}
Let $T:X\to Y$ be a bounded linear operator and $K:X\to Y$ a compact operator from two Banach spaces. Then $T+K$ is Fredholm with index $i(T+K)=i(T)$. 
\end{theorem}    
Now all the ingredients are supplied in order to prove the surjectivity of the operator.
\begin{theorem}\label{th surjectivity op}
The operator $\tfrac{1}{2}\mathbf{I}+\mathbf{K}+\mathbf{D}^R$ is onto in $\bm{L}^2(\partial C)$.
\end{theorem}

\begin{proof}
From Lemma \ref{lemma 1/2I+K} we have that the operator $\frac12\mathbf{I}+\mathbf{K}:\bm{L}^2(\partial C)\to \bm{L}^2(\partial C)$
is Fredholm with index $i\left(\frac12\mathbf{I}+\mathbf{K}\right)=0$, because both the nullity and the defect of this operator are null.
Moreover, since the operator $\mathbf{D}^R$ is compact from Lemma \ref{lemma D}, it follows by means of
Theorem \ref{index of a compact pertubation} that 
\begin{equation*}
	i\left(\tfrac{1}{2}\mathbf{I}+\mathbf{K}+\mathbf{D}^R\right)=0.
\end{equation*}
Hence
\begin{equation*}
	\textrm{dim}\left(\textrm{Ker}\left(\tfrac{1}{2}\mathbf{I}+\mathbf{K}+\mathbf{D}^R\right)\right)
	=\textrm{dim}\left(\mathbf{L}^2(\partial C)/\textrm{Im}\left(\tfrac{1}{2}\mathbf{I}+\mathbf{K}+\mathbf{D}^R\right)\right).
\end{equation*} 
Since the operator $\frac12\mathbf{I}+\mathbf{K}+\mathbf{D}^R$ is injective it shows that
$\textrm{dim}(\textrm{Ker}(\frac12\mathbf{I}+\mathbf{K}+\mathbf{D}^R))=0$.
Finally, $\textrm{dim}(\mathbf{L}^2(\partial C)/\textrm{Im}\left(\frac12\mathbf{I}+\mathbf{K}+\mathbf{D}^R)\right)=0$,
that is $\textrm{Im}\left(\tfrac{1}{2}\mathbf{I}+\mathbf{K}+\mathbf{D}^R\right)=\bm{L}^2(\partial C)$.
\end{proof}

Summing up, it follows
\begin{corollary}
There exists a unique solution to \eqref{direct problem}.
\end{corollary}
\begin{proof}
Uniqueness follows from Theorem \ref{th uniqueness op} and the existence from Theorem \ref{th surjectivity op}.
\end{proof}

\section{Rigorous derivation of the asymptotic expansion}

In this section, with the integral representation formula \eqref{representation formula} at hand, we consider the hypothesis that the
cavity $C$ is small compared to the distance from the boundary of the half-space.
The aim is to derive an asymptotic expansion of the solution $\bm{u}$. 

In particular, let us take the cavity, that from now on we denote by $C_{\varepsilon}$ to highlight the dependence from $\varepsilon$, as 
\begin{equation*}
C_{\varepsilon}=\bm{z}+\varepsilon \varOmega
\end{equation*} 
where $\varOmega$ is a bounded Lipschitz domain containing the origin. At the same time, we write the solution of the boundary value problem \eqref{direct problem} as $\bm{u}_{\varepsilon}$.
From \eqref{representation formula}, recalling that $\mathbf{N}=\mathbf{\Gamma}+\mathbf{R}$, we have
\begin{equation}\label{repres formula with Ceps}
\begin{aligned}
u^{k}_{\varepsilon}(\bm{y})&=p \int\limits_{\partial C_{\varepsilon}}\bm{N}^{(k)}(\bm{x},\bm{y})\cdot \bm{n}(\bm{x})\, d\sigma(\bm{x})-\int\limits_{\partial C_{\varepsilon}}\frac{\partial\bm{N}^{(k)}}{\partial\bm{\nu}}(\bm{x},\bm{y})\cdot\bm{f}(\bm{x})\, d\sigma(\bm{x})\\
&:=I^{(k)}_1(\bm{y})+I^{(k)}_2(\bm{y}),\quad \bm{y}\in\mathbb{R}^2
\end{aligned}
\end{equation}
for $k=1,2,3$, where $u^{k}_{\varepsilon}$ indicates the $k$-th component of the displacement vector and $\bm{f}=\bm{f}_{\varepsilon}$ is the solution of \eqref{trace equation}, that is
\begin{equation}\label{int equ on Ceps}
\left(\tfrac{1}{2}\mathbf{I}+\mathbf{K}_{\varepsilon}+\mathbf{D}_{\varepsilon}^R\right)\bm{f}_{\varepsilon}(\bm{x})=p\left(\mathbf{S}_{\varepsilon}^{\Gamma}(\bm{n})(\bm{x})+\mathbf{S}^{R}_{\varepsilon}(\bm{n})(\bm{x})\right),\qquad \bm{x}\in\partial C_{\varepsilon}
\end{equation}
where we add the dependence from $\varepsilon$ to all the layer potentials to distinguish them, in the sequel, from the layer potential defined over a domain indipendent from $\varepsilon$. In what follows, with $\mathbb{I}$ we indicate the fourth-order symmetric tensor such that $\mathbb{I}\mathbf{A}=\mathbf{\widehat{A}}$ and for any fixed value of $\varepsilon>0$, given $\bm{h}:\partial C_{\varepsilon}\to \mathbb{R}^3$, we introduce the function $\widehat{\bm{h}}:\partial\varOmega\to \mathbb{R}^3$ defined by
\begin{equation*}
	\widehat{\bm{h}}(\bm{\zeta}):=\bm{h}(\bm{z}+\varepsilon\bm{\zeta}),\qquad \bm{\zeta}\in\partial\varOmega.
\end{equation*}
Moreover, we consider the functions $\bm{\theta}^{qr}$, for $q,r=1,2,3$, solutions to 
\begin{equation}\label{fz theta}
\textrm{div}(\mathbb{C}\widehat{\nabla}\bm{\theta}^{qr})=0\quad \textrm{in}\, \mathbb{R}^3\setminus \varOmega,\qquad\quad
\frac{\partial \bm{\theta}^{qr}}{\partial\bm{\nu}}=-\frac{1}{3\lambda+2\mu}\mathbb{C}\bm{n}\quad \textrm{on}\, \partial\varOmega,
\end{equation}
with the decay conditions at infinity
\begin{equation}\label{decay theta}
|\bm{\theta}^{qr}|=O(|\bm{x}|^{-1}),\qquad\quad |\nabla\bm{\theta}^{qr}|=O(|\bm{x}|^{-2}),\qquad \textrm{as}\, |\bm{x}|\to \infty,
\end{equation} 
where the condition $\partial\bm{\theta}^{qr}/\partial\bm{\nu}$ has to be read as
\begin{equation*}
\left(\frac{\partial \bm{\theta}^{qr}}{\partial\bm{\nu}}\right)_i=-\frac{1}{3\lambda+2\mu}{C}_{ijqr}{n}_j
\end{equation*}
We now state our main result
\begin{theorem}[asymptotic expansion]\label{th asymp exp}
There exists $\varepsilon_0>0$ such that for all $\varepsilon\in(0,\varepsilon_0)$ at any $\bm{y}\in\mathbb{R}^2$ the following expansion holds
\begin{equation}\label{asymptotic expansion}
	u^{k}_{\varepsilon}(\bm{y})=\varepsilon^3|\varOmega| p\widehat{\nabla}_{\bm{z}}\bm{N}^{(k)}(\bm{z},\bm{y}):\mathbb{M}\mathbf{I}+O(\varepsilon^4),
\end{equation}
for $k=1,2,3$, where $O(\varepsilon^4)$ denotes a quantity bounded by $C\varepsilon^4$ for some uniform constant $C>0$, and $\mathbb{M}$ is the fourth-order elastic moment tensor defined by
\begin{equation}\label{elastic moment tensor}
\mathbb{M}:=\mathbb{I}+\frac{1}{|\varOmega|}\int\limits_{\partial\varOmega}\mathbb{C}(\bm{\theta}^{qr}(\bm{\zeta})\otimes \bm{n}(\bm{\zeta}))\, d\sigma(\bm{\zeta}),
\end{equation}  
where $\bm{\theta}^{qr}$, for $q,r=1,2,3$, solve the problem in \eqref{fz theta} and \eqref{decay theta}.
\end{theorem}

Before proving the theorem on the asymptotic expansion of $\bm{u}_{\varepsilon}$, we need to present some results

\begin{lemma}\label{1/2I+K+D inv}
The integral equation \eqref{int equ on Ceps}, when $\bm{x}=\bm{z}+\varepsilon\bm{\zeta}$, with $\bm{\zeta}\in\partial\varOmega$, is such that
\begin{equation}\label{trace equation lemma}
	\left(\tfrac{1}{2}\mathbf{I}+\mathbf{K}+\varepsilon^2\mathbf{\Lambda}_{\varOmega,\varepsilon}\right)\widehat{\bm{f}}(\bm{\zeta})
	=\varepsilon p \mathbf{S}^{\Gamma}(\bm{n})(\bm{\zeta})+O(\varepsilon^2),
\end{equation}
where
\begin{equation*}
	\mathbf{\Lambda}_{\varOmega,\varepsilon}\widehat{\bm{f}}(\bm{\eta}):=\int\limits_{\partial \varOmega}
	\frac{\partial\mathbf{R}}{\partial\bm{\nu}(\bm{\eta})}(\bm{z}+\varepsilon\bm{\eta},\bm{z}
	+\varepsilon\bm{\zeta})\widehat{\bm{f}}(\bm{\eta})\, d\sigma(\bm{\eta})
\end{equation*}
is uniformly bounded in $\varepsilon$.
Moreover, when $\varepsilon$ is sufficiently small, we have
\begin{equation*}
	\widehat{\bm{f}}(\bm{\zeta})=\varepsilon p\left(\tfrac{1}{2}\mathbf{I}+\mathbf{K}\right)^{-1}\mathbf{S}^{\Gamma}(\bm{n})(\bm{\zeta})
	+O(\varepsilon^2),\qquad \bm{\zeta}\in\partial\varOmega.
\end{equation*}
\end{lemma}

\begin{proof}
At the point $\bm{z}+\varepsilon\bm{\zeta}$, where $\bm{\zeta}\in\partial\varOmega$, we obtain
\begin{equation*}
\begin{aligned}
\mathbf{D}_{\varepsilon}^R\bm{f}(\bm{z}+\varepsilon\bm{\zeta})&=\int\limits_{\partial C_{\varepsilon}}\frac{\partial\mathbf{R}}{\partial\bm{\nu}(\bm{t})}(\bm{t},\bm{z}+\varepsilon\bm{\zeta})\bm{f}(\bm{t})\, d\sigma(\bm{t})\\
&=\varepsilon^2\int\limits_{\partial \varOmega}\frac{\partial\mathbf{R}}{\partial\bm{\nu}(\bm{\eta})}(\bm{z}+\varepsilon\bm{\eta},\bm{z}+\varepsilon\bm{\zeta})\widehat{\bm{f}}(\bm{\eta})\, d\sigma(\bm{\eta}).
\end{aligned}
\end{equation*}
Therefore, recalling that the kernel $\partial\mathbf{R}/\partial\bm{\nu}(\bm{\eta})$ is continuous we get
\begin{equation}\label{DReps}
\mathbf{D}_{\varepsilon}^R=\varepsilon^2 \mathbf{\Lambda}_{\varOmega,\varepsilon}
\end{equation}
where $\|\mathbf{\Lambda}_{\varOmega,\varepsilon}\|\leq C'$, with $C'$ indipendent from $\varepsilon$.\\
For the integral 
\begin{equation*}
\mathbf{K}_{\varepsilon}\bm{f}(\bm{z}+\varepsilon\bm{\zeta})=\textrm{p.v.}\int\limits_{\partial C_{\varepsilon}}\frac{\partial\mathbf{\Gamma}}{\partial\bm{\nu}(\bm{t})}(\bm{t}-\bm{z}-\varepsilon\bm{\zeta})\bm{f}(\bm{t})\, d\sigma(\bm{t})
\end{equation*}
we use the explicit expression of the conormal derivative of the fundamental solution of the Lam\'e operator given in \eqref{traction Gamma}. In particular, since \eqref{traction Gamma} is a homogeneous function of degree -2, with the substitution $\bm{t}=\bm{z}+\varepsilon\bm{\eta}$, we find
\begin{equation*}
\begin{aligned}
	\left(\frac{\partial\bm{\Gamma}^{(k)}}{\partial\bm{\nu}}\right)_h(\varepsilon(\bm{\eta}-\bm{\zeta}))=&-\frac{1}{4\pi\varepsilon^2}\Bigg\{\left[\frac{1-2\nu}{2(1-\nu)}\delta_{hk}+\frac{3}{2(1-\nu)}\frac{\eta_k-\zeta_k}{|\bm{\eta}-\bm{\zeta}|}\frac{\eta_h-\zeta_h}{|\bm{\eta}-\bm{\zeta}|}\right]\frac{\partial}{\partial\bm{n}(\bm{\eta})}\frac{1}{|\bm{\eta}-\bm{\zeta}|}\\
	&+\frac{1-2\nu}{2(1-\nu)}n_h(\bm{\eta})\frac{\eta_k-\zeta_k}{|\bm{\eta}-\bm{\zeta}|^3}-\frac{1-2\nu}{2(1-\nu)}n_k(\bm{\eta})\frac{\eta_h-\zeta_h}{|\bm{\eta}-\bm{\zeta}|^3}\Bigg\}\\
	=&\frac{1}{\varepsilon^2}\left(\frac{\partial\bm{\Gamma}^{(k)}}{\partial\bm{\nu}}\right)_h(\bm{\eta}-\bm{\zeta}),
\end{aligned}
\end{equation*}
for $h,k=1,2,3$.
Therefore, we immediately obtain that
\begin{equation}\label{KGammaeps}
\mathbf{K}_{\varepsilon}\bm{f}(\bm{z}+\varepsilon\bm{\zeta})=\textrm{p.v.}\int\limits_{\partial \varOmega}\frac{\partial\bm{\Gamma}}{\partial\bm{\nu}(\bm{\eta})}(\bm{\eta}-\bm{\zeta})\widehat{\bm{f}}(\bm{\eta})\, d\sigma(\bm{\eta})=\mathbf{K}\widehat{\bm{f}}(\bm{\zeta}).
\end{equation}
Evaluating the other integrals in \eqref{int equ on Ceps} we obtain
\begin{equation*}
\mathbf{S}^{\Gamma}_{\varepsilon}(\bm{n})(\bm{z}+\varepsilon\bm{\zeta})=\int\limits_{\partial C_{\varepsilon}}\mathbf{\Gamma}(\bm{t}-\bm{z}-\varepsilon\bm{\zeta})\bm{n}(\bm{t})\, d\sigma(\bm{t})
\end{equation*}
hence, choosing $\bm{t}=\bm{z}+\varepsilon\bm{\eta}$, with $\bm{\eta}\in\partial\varOmega$, we find
\begin{equation}\label{SGammaeps}
	\mathbf{S}^{\Gamma}_{\varepsilon}(\bm{n})(\bm{z}+\varepsilon\bm{\zeta})
	=\varepsilon^2\int\limits_{\partial \varOmega}\mathbf{\Gamma}(\varepsilon(\bm{\eta}-\bm{\zeta}))\bm{n}(\bm{\eta})\, d\sigma(\bm{\eta})
	=\varepsilon \mathbf{S}^{\Gamma}(\bm{n})(\bm{\zeta}),
\end{equation}
where the last equality follows noticing that the fundamental solution  is homogeneous of degree -1.
In a similar way
\begin{equation*}
	\mathbf{S}^R_{\varepsilon}(\bm{n})(\bm{z}+\varepsilon\bm{\zeta})
	=\int\limits_{\partial C_{\varepsilon}}\mathbf{R}(\bm{t},\bm{z}+\varepsilon\bm{\zeta})\bm{n}(\bm{t})\, d\sigma(\bm{t}),
\end{equation*}
hence, taking again $\bm{t}=\bm{z}+\varepsilon\bm{\eta}$, we find
\begin{equation*}
	\mathbf{S}^R_{\varepsilon}(\bm{n})(\bm{z}+\varepsilon\bm{\zeta})
	=\varepsilon^2 \int\limits_{\partial \varOmega}\mathbf{R}(\bm{z}+\varepsilon\bm{\eta},\bm{z}
	+\varepsilon\bm{\zeta})\bm{n}(\bm{\eta})\, d\sigma(\bm{\eta})
\end{equation*}
and since $\mathbf{R}$ is regular it follows that 
\begin{equation}\label{SReps}
	\mathbf{S}^R_{\varepsilon}(\bm{n})(\bm{z}+\varepsilon\bm{\zeta})=O(\varepsilon^2).
\end{equation}
Relation \eqref{trace equation lemma} follows putting together the result in \eqref{DReps}, \eqref{KGammaeps}, \eqref{SGammaeps} and \eqref{SReps}.

To conclude, from \eqref{trace equation lemma} we have
\begin{equation*}
	\left(\tfrac{1}{2}\mathbf{I}+\mathbf{K}\right)\left(\mathbf{I}
	+\varepsilon^2\left(\tfrac{1}{2}\mathbf{I}+\mathbf{K}\right)^{-1}\mathbf{\Lambda}_{\varepsilon,\varOmega}\right)\widehat{\bm{f}}
	=\varepsilon p\mathbf{S}^{\Gamma}(\bm{n})+O(\varepsilon^2),\qquad \textrm{on}\, \ \partial\varOmega.
\end{equation*}
From Lemma \ref{lemma 1/2I+K} and the continuous property of $\mathbf{\Lambda}_{\varepsilon,\varOmega}$ described before, we have
\begin{equation*}
	\Big\|\left(\tfrac{1}{2}\mathbf{I}+\mathbf{K}\right)^{-1}\mathbf{\Lambda}_{\varepsilon,\varOmega}\Big\|\leq C
\end{equation*}
where $C>0$ is indipendent from $\varepsilon$. On the other hand, choosing $\varepsilon_0^2=1/2C$, it follows that for all $\varepsilon\in(0,\varepsilon_0)$ we have
\begin{equation*}
	\mathbf{I}+\varepsilon^2\left(\tfrac{1}{2}\mathbf{I}+\mathbf{K}\right)^{-1}\mathbf{\Lambda}_{\varepsilon,\varOmega}
\end{equation*}
is invertible and 
\begin{equation*}
\left(\mathbf{I}+\varepsilon^2\left(\tfrac{1}{2}\mathbf{I}+\mathbf{K}\right)^{-1}\mathbf{\Lambda}_{\varepsilon,\varOmega}\right)^{-1}=\mathbf{I}+O(\varepsilon^2).
\end{equation*}
Therefore
\begin{equation*}
	\widehat{\bm{f}}=\varepsilon p \left(\tfrac{1}{2}\mathbf{I}+\mathbf{K}\right)^{-1}\mathbf{S}^{\Gamma}(\bm{n})
	+O(\varepsilon^2),\qquad \textrm{on}\, \ \partial\varOmega,
\end{equation*}
that is the assertion.
\end{proof}

For ease of reading, we define the function $\bm{w}:\partial\varOmega\to \partial\varOmega$ as
\begin{equation}\label{w definition}
	\bm{w}(\bm{\zeta}):=-\left(\tfrac{1}{2}\mathbf{I}+\mathbf{K}\right)^{-1}\mathbf{S}^{\Gamma}(\bm{n})(\bm{\zeta}),
	\qquad \bm{\zeta}\in\partial\varOmega.
\end{equation}
Taking the problem 
\begin{equation}\label{external problem}
	\textrm{div}\left(\mathbb{C}\widehat{\nabla}\bm{v}\right)=0\, \qquad\ \textrm{in}\,\mathbb{R}^3\setminus\varOmega,
	\qquad\qquad \frac{\partial\bm{v}}{\partial\bm{\nu}}=-\bm{n}\,\qquad\ \textrm{on}\, \partial\varOmega
\end{equation}
with decay conditions at infinity 
\begin{equation}\label{bound cond external problem}
	\bm{v}=O(|\bm{x}|^{-1}),\,\qquad\qquad |\nabla\bm{v}|=O(|\bm{x}|^{-2})\,\qquad \textrm{as}\,\ |\bm{x}|\to +\infty,
\end{equation}
we show that $\bm{w}(\bm{x})$,
for $\bm{x}\in\partial\varOmega$, is the trace of $\bm{v}$ on the boundary of $\varOmega$.
The well-posedness of this problem is a classical result in the theory of linear elasticity so we remind the reader,
for example,  to \cite{Fichera,Gurtin,Kupradze}.

\begin{proposition}
The function $\bm{w}$, defined in \eqref{w definition}, is such that $\bm{w}=\bm{v}\big|_{\bm{x}\in\partial\varOmega}$
where $\bm{v}$ is the solution to \eqref{external problem} and \eqref{bound cond external problem}.
\end{proposition}

\begin{proof}
Applying second Betti's formula to the fundamental solution $\bm{\Gamma}$ and the function $\bm{v}$ into the domain
$B_{r}(\bm{0})\setminus (\varOmega\cup B_{\varepsilon}(\bm{x}))$, with $\varepsilon>0$ and $r>0$ sufficiently large such
that to contain the cavity $\varOmega$, we obtain, as done in a similar way in the proof of the Theorem \ref{th of repr formula},
\begin{equation*}
	\bm{v}(\bm{x})=-\mathbf{S}^{\Gamma}\bm{n}(\bm{x})-\mathbf{D}^{\Gamma}\bm{v}(\bm{x}),\qquad
	\bm{x}\in\mathbb{R}^3\setminus \varOmega
\end{equation*}
Therefore, from the single and double layer potential properties for the elastostatic equations, we find
\begin{equation*}
	\bm{v}(\bm{x})=-\mathbf{S}^{\Gamma}\bm{n}(\bm{x})-\left(-\tfrac{1}{2}\mathbf{I}+\mathbf{K}\right)\bm{v}(\bm{x}),
	\qquad \bm{x}\in\partial\varOmega
\end{equation*}
hence
\begin{equation*}
	\bm{v}(\bm{x})=-\left(\tfrac{1}{2}\mathbf{I}+\mathbf{K}\right)^{-1}\mathbf{S}^{\Gamma}(\bm{n})(\bm{x}),
	\qquad \bm{x}\in\partial\varOmega
\end{equation*}
that is the assertion.
\end{proof}   

We note that the function $\bm{v}$, as well as its trace $\bm{w}$ on $\partial\varOmega$, can be written in terms of the functions $\bm{\theta}^{qr}$. Indeed, taking
\begin{equation*}
\bm{v}=\bm{\theta}^{qr}\delta_{qr},\qquad q,r=1,2,3,
\end{equation*}
and using \eqref{fz theta} and \eqref{decay theta}, it is straightforward to see that the elastostatic equation and the boundary condition in \eqref{external problem} are satisfied.

\begin{proof}[Proof of Theorem \ref{th asymp exp}]
We study separately the two integrals $I^{(k)}_1, I^{(k)}_2$ defined in \eqref{repres formula with Ceps}. Since $\bm{y}\in\mathbb{R}^2$ and $\bm{x}\in\partial C_{\varepsilon}=\bm{z}+\varepsilon\bm{\zeta}$, with $\bm{\zeta}\in\partial \varOmega$, we consider the Taylor expansion for the Neumann function, that is
\begin{equation}\label{N taylor}
\bm{N}^{(k)}(\bm{z}+\varepsilon\bm{\zeta},\bm{y})=\bm{N}^{(k)}(\bm{z},\bm{y})+\varepsilon\nabla \bm{N}^{(k)}(\bm{z},\bm{y})\bm{\zeta}+O(\varepsilon^2),
\end{equation} 
for $k=1,2,3$. By the change of variable $\bm{x}=\bm{z}+\varepsilon\bm{\zeta}$ and substituting \eqref{N taylor} in $I^{(k)}_1$, we find
\begin{equation*}
\begin{aligned}
I^{(k)}_1&=\varepsilon^2p \bm{N}^{(k)}(\bm{z},\bm{y})\cdot\int\limits_{\partial\varOmega}\bm{n}(\bm{{\zeta}})\, d\sigma(\bm{\zeta})+\varepsilon^3p\int\limits_{\partial\varOmega}\bm{n}(\bm{\zeta})\cdot\nabla\bm{N}^{(k)}(\bm{z},\bm{y})\bm{\zeta}\, d\sigma(\bm{\zeta})+O(\varepsilon^4)\\
&:=p\left(\varepsilon^2I^{(k)}_{11}+\varepsilon^3I^{(k)}_{12}\right)+O(\varepsilon^4).
\end{aligned}
\end{equation*}
Integral $I^{(k)}_{11}$ is null, in fact, applying the divergence theorem
\begin{equation*}
\int\limits_{\partial\varOmega}\bm{n}(\bm{\zeta})\, d\sigma(\bm{\zeta})=0.
\end{equation*}
For the integral $I^{(k)}_{12}$, we use the equality $\bm{n}\cdot\nabla\bm{N}^{(k)}\bm{\zeta}=\nabla\bm{N}^{(k)}:(\bm{n}(\bm{\zeta})\otimes \bm{\zeta})$, therefore
\begin{equation}\label{I1 in asymp}
	I^{(k)}_{1}=\varepsilon^3p\nabla \bm{N}^{(k)}(\bm{z},\bm{y}):\int\limits_{\partial\varOmega}\bigl(\bm{n}(\bm{\zeta})\otimes \bm{\zeta}\bigr)\, d\sigma(\bm{\zeta})+O(\varepsilon^4),\qquad k=1,2,3.
\end{equation}
For the term $I^{(k)}_2$ we use the result in Lemma \ref{1/2I+K+D inv} and the Taylor expansion of the conormal derivative of $\mathbf{N}^{(k)}(\bm{x},\bm{y})$, for $k=1,2,3$. In particular, for $\bm{x}=\bm{z}+\varepsilon\bm{\zeta}$, when $\bm{\zeta}\in\partial\varOmega$ and $\bm{y}\in\mathbb{R}^2$, we consider only the first term of the asymptotic expansion, that is 
\begin{equation*}
	\frac{\partial \bm{N}^{(k)}}{\partial\bm{\nu}(\bm{x})}(\bm{x},\bm{y})
	=\frac{\partial \bm{N}^{(k)}}{\partial\bm{\nu}(\bm{\zeta})}(\bm{z},\bm{y})+O(\varepsilon),\qquad k=1,2,3.
\end{equation*}   
Therefore
\begin{equation*}
	\begin{aligned}
	I^{(k)}_2&=-\varepsilon^2\int\limits_{\partial\varOmega}\frac{\partial\bm{N}^{(k)}}{\partial\bm{\nu}(\bm{x})}(\bm{z+\varepsilon\bm{\zeta}},\bm{y})
		\cdot \widehat{\bm{f}}(\bm{\zeta})\, d\sigma(\bm{\zeta})\\
	&=-\varepsilon^3p\int\limits_{\partial\varOmega}\frac{\partial\bm{N}^{(k)}}{\partial\bm{\nu}(\bm{\zeta})}(\bm{z},\bm{y})
		\cdot \bm{w}(\bm{\zeta})\, d\sigma(\bm{\zeta})
		+O(\varepsilon^4),
	\end{aligned}
\end{equation*}
for any $k$, where $\bm{w}$ is defined in \eqref{w definition}. 
Since
$\partial\bm{N}^{(k)}/\partial\bm{\nu}(\bm{\zeta})=\mathbb{C}\widehat{\nabla}\bm{N}^{(k)}\bm{n}(\bm{\zeta})$, we have
\begin{equation*}
	\mathbb{C}\widehat{\nabla}\bm{N}^{(k)}\bm{n}(\bm{\zeta})\cdot \bm{w}(\bm{\zeta})
	=\mathbb{C}\widehat{\nabla}\bm{N}^{(k)}:(\bm{w}(\bm{\zeta})\otimes\bm{n}(\bm{\zeta})).
\end{equation*}
Therefore
\begin{equation}\label{I2 in asymp}
	I^{(k)}_2(\bm{y})=\varepsilon^3\,p\,\mathbb{C}\widehat{\nabla}\bm{N}^{(k)}(\bm{z},\bm{y})
	:\int\limits_{\partial\varOmega}(\bm{w}(\bm{\zeta})\otimes\bm{n}(\bm{\zeta}))\, d\sigma(\bm{\zeta})+O(\varepsilon^4).
\end{equation}
Collecting the result in \eqref{I1 in asymp} and \eqref{I2 in asymp}, equation \eqref{repres formula with Ceps} becomes
\begin{equation*}
	\begin{aligned}
	u^{k}_{\varepsilon}(\bm{y})&=I^{(k)}_1(\bm{y})+I^{(k)}_2(\bm{y})\\
	&\hspace{-1cm}=\varepsilon^3p\Bigg[\nabla\bm{N}^{(k)}(\bm{z},\bm{y}):\int\limits_{\partial\varOmega}(\bm{n}\otimes\bm{\zeta})\,d\sigma(\bm{\zeta})
		+\mathbb{C}\widehat{\nabla}\bm{N}^{(k)}(\bm{z},\bm{y}):\int\limits_{\partial\varOmega}(\bm{w}\otimes\bm{n})\, d\sigma(\bm{\zeta})\Bigg]		
		+O(\varepsilon^4).
	\end{aligned}
\end{equation*}
Now, handling this expression, we higlight the moment elastic tensor. We have
\begin{equation}\label{zeta otimes n}
	\int_{\partial \varOmega}\left(\bm{n}(\bm{\zeta})\otimes \bm{\zeta}\right) \,d\sigma(\bm{\zeta})=|\varOmega| \mathbf{I},
\end{equation}
indeed, for any $i,j=1,2,3$, it follows
\begin{equation*}
	\begin{aligned}
	\int_{\partial \varOmega}\zeta_i\,n_j\,d\sigma(\bm{\zeta})
	&=\int_{\partial \varOmega}\bm{n} \cdot  \zeta_i \bm{e}_j\, d\sigma(\bm{\zeta})\\
	&=\int_{\varOmega}\textrm{div}\left(\zeta_i\bm{e}_j\right)d\bm{\zeta}
	=\int_{\varOmega}\bm{e}_j\cdot \bm{e}_i\,d\bm{\zeta}=|\varOmega|\delta_{ij},
\end{aligned}
\end{equation*}
where $\bm{e}_j$ is the $j$-th unit vector of $\mathbb{R}^3$. Hence, by \eqref{zeta otimes n} and taking the symmetric
part of $\nabla\bm{N}^{(k)}$, for any $k$, we find
\begin{equation*}
	u^{k}_{\varepsilon}=\varepsilon^3 p\Bigg[\widehat{\nabla} \bm{N}^{(k)}:\mathbf{I}|\varOmega|+\mathbb{C}\widehat{\nabla}\bm{N}^{(k)}
	:\int\limits_{\partial \varOmega}\bm{w}\otimes\bm{n}\, d\sigma(\bm{\zeta})\Bigg]+O(\varepsilon^4).
\end{equation*}  
Using the symmetries of $\mathbb{C}$, we have
\begin{equation*}
	u^{k}_{\varepsilon}=\varepsilon^3|\varOmega| p\widehat{\nabla} \bm{N}^{(k)}:\Bigg[\mathbf{I}+\frac{1}{|\varOmega|}
	\int\limits_{\partial \varOmega}\mathbb{C}(\bm{w}\otimes\bm{n})\, d\sigma(\bm{\zeta})\Bigg]+O(\varepsilon^4),
\end{equation*}
for $k=1,2,3$. Now, taking into account that $\mathbf{I}=\mathbb{I}\mathbf{I}$ and using the equality $\bm{w}=\bm{\theta}^{qr}\delta_{qr}$, with $q,r=1,2,3$, we have the assertion.
\end{proof}

\subsection{The Mogi model}

In this subsection, starting from the asymptotic expansion \eqref{asymptotic expansion}, that is
\begin{equation*}
	u^{k}_{\varepsilon}(\bm{y})=\varepsilon^3|\varOmega| p\widehat{\nabla}_{\bm{z}}\bm{N}^{(k)}(\bm{z},\bm{y}):\mathbb{M}\mathbf{I}
		+O(\varepsilon^4),\qquad k=1,2,3,
\end{equation*} 
where $\mathbb{M}$ is the tensor given in \eqref{elastic moment tensor}, we recover the Mogi model, presented within the Section 2, related to a spherical cavity. We first recall that 
\begin{equation*}
	\mathbb{M}\mathbf{I}=\left[\mathbb{I}+\frac{1}{|\varOmega|}
	\int\limits_{\partial \varOmega}\mathbb{C}(\bm{\theta}^{qr}\otimes\bm{n})\, d\sigma(\bm{\zeta})\right]\mathbf{I}=\mathbf{I}+\frac{1}{|\varOmega|}
		\int\limits_{\partial \varOmega}\mathbb{C}(\bm{w}\otimes\bm{n})\, d\sigma(\bm{\zeta}),
\end{equation*} 
where, in the last equality, we use the link between the functions $\bm{w}$ and $\bm{\theta}^{qr}$ that is $\bm{w}=\bm{\theta}^{qr}\delta_{qr}$, $q,r=1,2,3$.
Therefore, to get the Mogi's formula, we first find the explicit expression of $\bm{w}$
when the cavity $\varOmega$ is the unit sphere and then we calculate the gradient of the Neumann function $\mathbf{N}$. 

We recall that $\bm{w}$ is the trace on the boundary of the cavity of the solution to the external
problem
\begin{equation*}
	\textrm{div}(\mathbb{C}\widehat{\nabla}\bm{v})=0\quad\textrm{ in }\mathbb{R}^3\setminus B_1(\bm{0}),\qquad
	\frac{\partial\bm{v}}{\partial\nu}=-\bm{n}\quad\textrm{ on }\partial B_1(\bm{0}),
\end{equation*}
where $B_1(\bm{0})=\{\bm{x}\in\mathbb{R}^3\,:\,|\bm{x}|\leq 1\}$ with decay at infinity
\begin{equation*}
	\bm{v}=O(|\bm{x}|^{-1}),\,\qquad\qquad |\nabla\bm{v}|=O(|\bm{x}|^{-2})\,\qquad \textrm{as}\,\ |\bm{x}|\to +\infty.
\end{equation*}
We look for a solution with the form
\begin{equation*}
	\bm{v}(\bm{x})=\phi(r)\,\bm{x}\qquad\qquad\textrm{with }r:=|\bm{x}|,
\end{equation*}
so that
\begin{equation*}
	\Delta v_i=\left\{\phi''+\frac{4\phi'}{r}\right\}x_i,\qquad
	\textrm{div}\,\bm{v}=r\phi'+3\phi,\qquad
	\nabla\textrm{div}\,\bm{v}=\left\{\phi''+\frac{4\phi'}{r}\right\}\bm{x}.
\end{equation*}
By direct substitution, since $\bm{n}=\bm{x}$ on $\partial B$, we get
\begin{equation*}
	\begin{aligned}
	\textrm{div}(\mathbb{C}\widehat{\nabla}\bm{v})&=(\lambda+2\mu)\left(\phi''+\frac{4\phi'}{r}\right)\bm{x},\\
	\frac{\partial\bm{v}}{\partial\nu}&=\bigl\{(\lambda+2\mu)r\phi'+(3\lambda+2\mu)\phi\bigr\}\bm{x}\\
	\end{aligned}
\end{equation*}
Thus, we need to find a function $\phi\,:\,[1,+\infty)\to\mathbb{R}$ such that
\begin{equation*}
	\phi''+\frac{4\phi'}{r}=0,\quad
	(\lambda+2\mu)r\phi'+(3\lambda+2\mu)\phi\bigr|_{r=1}=-1,\quad
	\phi\bigr|_{r=+\infty}=0.
\end{equation*}
Condition at infinity implies that $B=0$ and $A=1/4\mu$.
Therefore, the solution is $\mathbf{v}(\bm{x})=\bm{x}/4\mu|\bm{x}|^3$, which implies that
\begin{equation*}
	\bm{w}(\bm{x}):=\bm{v}(\bm{x})\Bigr|_{|\bm{x}|=1}=\frac{\bm{x}}{4\mu}.
\end{equation*}
With the function $\bm{w}$ at hand, we have that
\begin{equation*}
	\mathbf{I}+\frac{1}{|B_1(\bm{0})|}\int\limits_{\partial B_1(\bm{0})}\mathbb{C}(\bm{w}(\bm{\zeta})
		\otimes\bm{n}(\bm{\zeta}))\, d\sigma(\zeta)=\mathbf{I}+\frac{3}{16\pi\mu}\int\limits_{\partial B_1(\bm{0})}
	\frac{\mathbb{C}(\bm{\zeta}\otimes\bm{\zeta})}{|\bm{\zeta}|^3}\, d\sigma(\bm{\zeta}).
\end{equation*}
Through the use of spherical coordinates and ortogonality relations for the circular functions, it holds
\begin{equation*}
	\int\limits_{\partial B_1(\bm{0})}\frac{\bm{\zeta}\otimes\bm{\zeta}}{|\bm{\zeta}|^3}\, d\sigma(\bm{\zeta})=\frac{4\pi}{3}\mathbf{I},
\end{equation*} 
hence the second-order tensor $\mathbb{M}\mathbf{I}$ is given by
\begin{equation*}
	\mathbb{M}\mathbf{I}=\frac{3(\lambda+2\mu)}{4\mu}\mathbf{I}.
\end{equation*}
It implies
\begin{equation}\label{towards Mogi}
	u^{k}_{\varepsilon}(\bm{y})=\frac{\pi(\lambda+2\mu)}{\mu}\varepsilon^3 p\,
	\textrm{Tr}(\widehat{\nabla}_{\bm{z}}\bm{N}^{(k)}(\bm{z},\bm{y}))+O(\varepsilon^4),\qquad k=1,2,3,
\end{equation}
For the Neumann's function $\mathbf{N}$ (see the appendix for its explicit expression),
we are interested only to the trace of $\nabla_{\bm{z}} \mathbf{N}(\bm{z},\bm{y})$
computed at $y_3=0$.

Evaluating  $\mathbf{N}=\mathbf{N}(\bm{z},\bm{y})$ at $y_3=0$, we get
\begin{equation*}
	\begin{aligned}
	\kappa_{\mu}^{-1}N_{\alpha\alpha}&=-f-(z_\alpha-y_\alpha)^2 f^3-(1-2\nu)g+(1-2\nu)(z_\alpha-y_\alpha)^2 f g^2\\
	\kappa_{\mu}^{-1}N_{\beta\alpha}&=(z_\alpha-y_\alpha)(z_\beta-y_\beta)
		\bigl\{-f^3+(1-2\nu)fg\bigr\}\\
	\kappa_{\mu}^{-1}N_{3\alpha}&=(z_\alpha-y_\alpha)\bigl\{-z_3f^3+(1-2\nu)f g\bigr\}\\
	\kappa_{\mu}^{-1}N_{\alpha 3}&=(z_\alpha-y_\alpha)\Bigl\{-z_3f^3-(1-2\nu)fg\Bigr\}\\
	\kappa_{\mu}^{-1}N_{3 3}&=-2(1-\nu)f-z_3^2 f^3
	\end{aligned}
\end{equation*}
where $\alpha, \beta=1,2$ and $\kappa_{\mu}=1/(4\pi\mu)$,
with $f=1/|\bm{z}-\bm{y}|$ and $g=1/\bigl\{|\bm{z}-\bm{y}|-z_3\bigr\}$.

Let $\rho^2:=(z_1-y_1)^2+(z_2-y_2)^2$.
Using the identities
\begin{equation*}
	\rho^2 f^2=1-z_3^2 f^2,\qquad (1-z_3 f)g=f
\end{equation*}
and the differentiation formulas
\begin{equation*}
	\begin{aligned}
	&\partial_{\bm{z}_\alpha}f=-(z_\alpha-y_\alpha)f^3, &\quad &\partial_{\bm{z}_3}f=-z_3 f^3 \\
	&\partial_{\bm{z}_\alpha}g=-(z_\alpha-y_\alpha)fg, &\quad &\partial_{\bm{z}_3}g=fg,\\
	&\partial_{\bm{z}_\alpha}(fg)=-(z_\alpha-y_\alpha)(f+g)f^2 g, &\quad
		&\partial_{\bm{z}_3}(fg)=f^3,
	\end{aligned}
\end{equation*}
we deduce the following formulas for some of the derivatives of $\kappa_\mu^{-1} N_{ij}$
\begin{equation*}
	\begin{aligned}
	\kappa_{\mu}^{-1}\partial_{\bm{z}_\alpha}N_{\alpha \alpha}&=(z_\alpha-y_\alpha)
		\bigl\{-f^3+3(z_\alpha-y_\alpha)^2f^5+(1-2\nu)\bigl[3f-(z_\alpha-y_\alpha)^2f^2(f+2g)\bigr]g^2\bigr\}\\
	\kappa_{\mu}^{-1}\partial_{\bm{z}_\beta}N_{\beta \alpha}&=(z_\alpha-y_\alpha)
		\bigl\{-f^3+3(z_\beta-y_\beta)^2f^5+(1-2\nu)\bigl[f-(z_\beta-y_\beta)^2f^2(f+2g)\bigr]g^2\bigr\}\\
	\kappa_{\mu}^{-1}\partial_{\bm{z}_3}N_{3 \alpha}&=(z_\alpha-y_\alpha)
		\bigl\{-2\nu f^3+3z_3^2 f^5\bigr\}\\
	\kappa_{\mu}^{-1}\partial_{\bm{z}_\alpha}N_{\alpha 3}&=-z_3 f^3+3(z_\alpha-y_\alpha)^2z_3 f^5
		+(1-2\nu)\bigl[-1+(z_\alpha-y_\alpha)^2(f+g)f\bigr]fg\\
	\kappa_{\mu}^{-1}\partial_{\bm{z}_3}N_{3 3}&=-2\nu z_3 f^3+3z_3^3 f^5.
	\end{aligned}
\end{equation*}
As a consequence, we obtain
\begin{equation}\label{trace grad N}
	\begin{aligned}
	\textrm{Tr}\bigl(\hat{\nabla} N^{(\alpha)}\bigr)&=2\kappa_{\mu}(1-2\nu)(z_\alpha-y_\alpha)f^3,
		\qquad \textrm{for}\, \alpha=1,2\\
	\textrm{Tr}\bigl(\hat{\nabla} N^{(3)}\bigr)&=2\kappa_{\mu}(1-2\nu)z_3 f^3.
	\end{aligned}
\end{equation}
Combining \eqref{towards Mogi}, \eqref{trace grad N} and using the explicit expression for $f$, we find
\begin{equation*}
	\begin{aligned}
	u^{\alpha}_{\varepsilon}(\bm{y})&=\frac{1-\nu}{\mu}\frac{\varepsilon^3p (z_{\alpha}-y_{\alpha})}{|\bm{z}-\bm{y}|^3}
		+O(\varepsilon^4),\qquad \textrm{for}\, \alpha=1,2 \\ 
	u^{3}_{\varepsilon}(\bm{y})&=\frac{1-\nu}{\mu}\frac{\varepsilon^3p\,z_3}{|\bm{z}-\bm{y}|^3}
		+O(\varepsilon^4),\vspace{1cm}
	\end{aligned}
\end{equation*}
that are the components given in \eqref{mogi}.

\section*{Appendix. Neumann function for the half-space with zero traction}

For reader's convenience, we provide here the complete derivation of the explicit formula for the
Neumann function $\mathbf{N}$ of the problem 
\begin{equation*}
	\mathcal{L}\bm{v}:=\text{div}\bigl(\mathbb{C}\widehat{\nabla}\bm{v}\bigr)=\bm{b}\quad \textrm{in }\mathbb{R}^3_-,
	\qquad\qquad
	\bigl(\mathbb{C}\widehat{\nabla}\bm{v}\bigr)\bm{e}_3=\bm{0}\quad \textrm{in }\mathbb{R}^2
\end{equation*}
as stated in Theorem \ref{thm:fundsol}.
By definition, $\mathbf{N}$ is the kernel of the integral operator which associate to the forcing term
$\bf{b}$ the solution $\bm{v}$ to the boundary value problem, viz.
\begin{equation*}
	\bm{v}(\bm{x})=\int_{\mathbb{R}^3_-}\mathbf{N}(\bm{x},\bm{y})\bm{b}(\bm{y})\, d\bm{y}.
\end{equation*}
The explicit expression for the Neumann function has been determined by Mindlin in \cite{Mindlin36,Mindlin54}
using different approaches.
Here, we follow the second one which is based on the Papkovich--Neuber representation of the displacement $\bm{v}$.

Let us introduce the functions
\begin{equation*}
	\phi(\bm{x}):=\frac{1}{|\bm{x}|}\qquad\textrm{and}\qquad
	\psi(\bm{x}):=\frac{\phi(\bm{x})}{1-x_3\phi(\bm{x})}=\frac{1}{|\bm{x}|-x_3}.
\end{equation*}
observing that, apart for $\partial_i \phi=-x_i\phi^3$, $i=1,2,3$, the following identities hold true
for $\alpha=1,2$,
\begin{equation*}
	\phi-\psi=-x_3\phi\,\psi,\qquad
	\partial_\alpha \psi=-x_\alpha \phi\,\psi^2,\qquad
	\partial_3 \psi=\phi\,\psi,\qquad
	\partial_3(\phi\,\psi)=\phi^3.
\end{equation*}
We denote by $\phi$ and $\widetilde{\phi}$ the values $\phi(\bm{x}+\bm{e}_3)$
and $\phi(\bm{x}-\bm{e}_3)$, respectively, with analogous notation for $\psi$.

\begin{proposition}\label{prop:auxiliaryGreen}
Let $\mathbf{I}$ be the identity matrix and $\delta$ the Dirac delta concentrated at $-\bm{e}_3$.
Then, the matrix-valued function $\mathbf{\mathcal{N}}=\mathbf{\mathcal{N}}(\bm{x})$ solution to
\begin{equation*}
	\mathcal{L}\bm{v}:=\textrm{\rm div}\bigl(\mathbb{C}\widehat{\nabla}\bm{v}\bigr)=\delta \mathbf{I}\quad \textrm{in }\mathbb{R}^3_-,
	\qquad\qquad
	\bigl(\mathbb{C}\widehat{\nabla}\bm{v}\bigr)\bm{e}_3=\bm{0}\quad \textrm{in }\mathbb{R}^2,
\end{equation*}
is given by 
\begin{equation}\label{calNformulas}
	\begin{aligned}
	\mathcal{N}_{\alpha\alpha}&=-C_{\mu,\nu}\bigl\{(3-4\nu)\phi+x_\alpha^2\phi^3
		+\widetilde{\phi}+[(3-4\nu)x_\alpha^2-2x_3]\widetilde{\phi}^3+6x_\alpha^2x_3\widetilde{\phi}^5\\
	&\hskip9cm 	+c_\nu\bigl(\widetilde{\psi}
		-x_\alpha^2\widetilde{\phi}\,\widetilde{\psi}^2\bigr)\bigr\}\\
	\mathcal{N}_{\alpha \beta}&=-C_{\mu,\nu} x_\alpha x_\beta\bigl\{\phi^3+(3-4\nu)\widetilde{\phi}^3
		+6x_3\widetilde{\phi}^5-c_\nu\widetilde{\phi}\,\widetilde{\psi}^2\bigr\}\\
	\mathcal{N}_{3 \alpha}&=-C_{\mu,\nu} x_\alpha\bigl\{(x_3+1)\phi^3+(3-4\nu)(x_3+1)\widetilde{\phi}^3
		+6x_3(x_3-1)\widetilde{\phi}^5-c_\nu\widetilde{\phi}\,\widetilde{\psi}\bigr\}\\
	\mathcal{N}_{\alpha 3}&=-C_{\mu,\nu}\,x_\alpha\bigl\{(x_3+1)\phi^3+(3-4\nu)(x_3+1)\widetilde{\phi}^3
		-6x_3(x_3-1)\widetilde{\phi}^5+c_\nu\widetilde{\phi}\,\widetilde{\psi}\bigr\},\\
	\mathcal{N}_{3 3}&=-C_{\mu,\nu}\bigl\{(3-4\nu)\phi+(x_3+1)^2\phi^3+(1+c_\nu)\widetilde{\phi}
		+\bigl[(3-4\nu)(x_3-1)^2+2x_3\bigr]\widetilde{\phi}^3\\
	&\hskip9cm	-6x_3(x_3-1)^2\widetilde{\phi}^5\bigr\}
	\end{aligned}
\end{equation}
where $C_{\mu,\nu}:=1/\{16\pi\mu(1-\nu)\}$, $c_\nu:=4(1-\nu)(1-2\nu)$ and $\alpha=1,2$.
\end{proposition}

To establish \eqref{calNformulas}, we observe that the columns $\mathcal{N}^{(i)}$ of $\mathbf{\mathcal{N}}$ are determined
by solving the equation $\mathcal{L}\bm{v}=\bm{e}_i\delta$ for $i=1,2,3$ and
using the Papkovich--Neuber representation
\begin{equation}\label{vhbeta}
	\bm{v}=C_{\mu,\nu}\bigl\{4(1-\nu)\bm{h}-\nabla\bigl(\bm{x}\cdot\bm{h}+\beta\bigr)\bigr\}
	\quad\textrm{with }\left\{\begin{aligned}
		\Delta \bm{h}&=4\pi\bm{e}_i\delta\\ \Delta \beta&=4\pi \delta_{i3}\delta.
			\end{aligned}\right.
\end{equation}
where $\delta_{ij}$ is the Kronecker symbol.
The coupling between $\bm{h}$ and $\beta$ is determined by the boundary conditions on the plane
$\{x_3=0\}$, which are
\begin{equation}\label{bdaryhbeta}
	\begin{aligned}
	&(1-2\nu)(\partial_3 h_\alpha+\partial_\alpha h_3)-\bm{x}\cdot\partial^2_{\alpha 3}\bm{h}
		-\partial^2_{\alpha 3}\beta=0,\qquad (\alpha=1,2),\\
	&2\nu\,\textrm{div}\,\bm{h}+2(1-2\nu)\,\partial_3 h_3
		-\bm{x}\cdot\partial^2_{33}\bm{h}-\partial^2_{33}\beta=0,
	\end{aligned}
	\qquad\textrm{for }x_3=0.
\end{equation}
Given $\bm{y}:=(y_1,y_2,y_3)$, let $\widetilde{\bm{y}}$ be the reflexed point. 
Set 
\begin{equation*}
	G(\bm{x},\bm{y}):=-\phi\bigl(\bm{x}-\bm{y}\bigr)+\phi\bigl(\bm{x}-\widetilde{\bm{y}}\bigr).
\end{equation*}
Denoting by $\langle f,g\rangle$ the action of the distribution $f$ on the function $g$,
we determine $\bm{h}$ and $\beta$ taking advantage of the relation
(which descends from the second Green identity)
\begin{equation}\label{green}
	F(\bm{x})=\tfrac{1}{4\pi}\langle \Delta F, G(\bm{x},\cdot)\rangle,
\end{equation}
applied to different choices of $F$.

\begin{proof}[Proof of Proposition \ref{prop:auxiliaryGreen}.]
To determine $\mathcal{N}$, we consider separately the case of horizontal and vertical forcing.
By symmetry, $x_1$ and $x_2$ can be interchanged.
\vskip.25cm

{\bf Horizontal force: $\mathcal{L}\bm{v}=\bm{e}_1\delta$.}
We choose $h_2=0$, so that boundary conditions become
\begin{equation*}
	\left\{\begin{aligned}
	&(1-2\nu)(\partial_3 h_1+\partial_1 h_3)-x_1\partial^2_{13}h_1-\partial^2_{13}\beta=0,\\
	&(1-2\nu)\partial_2 h_3-x_1\partial^2_{23}h_1-\partial^2_{23}\beta=0,\\
	&2\nu\,\partial_1 h_1+2(1-\nu)\,\partial_3 h_3-x_1\partial^2_{33}h_1-\partial^2_{33}\beta=0,
	\end{aligned}\right.
	\qquad\textrm{for }x_3=0,
\end{equation*}
Differentiating the first equation with respect to $x_1$, the second with respect to $x_2$
and taking the difference, we obtain
\begin{equation*}
	0=(1-2\nu)\partial^2_{23} h_1+\partial^2_{23}h_1=2(1-\nu)\partial^2_{23} h_1
	\qquad\textrm{for }x_3=0,
\end{equation*}
which suggests, after integration with respect to $x_2$, the choice $F:=\partial_3 h_1$.
Applying \eqref{green},
\begin{equation*}
	\partial_3 h_1=-\partial_{y_3}G\bigr|_{\bm{y}=-\bm{e}_3}=-\partial_3(\phi+\widetilde{\phi}),
	\qquad\textrm{for }x_3<0,
\end{equation*}
and thus $h_1=-(\phi+\widetilde{\phi})$.

Being $\partial_3 h_1$ null for $x_3=0$, integration of the second boundary condition
encourages the choice  $F:=(1-2\nu)h_3-\partial_{3}\beta$ which is zero for $x_3=0$.
Hence, since $\Delta F=0$, we deduce
\begin{equation}\label{firsthor}
	(1-2\nu)h_3-\partial_{3}\beta=0,\qquad\textrm{for }x_3<0.
\end{equation}
Concerning the third boundary condition, we observe that
\begin{equation*}
	\begin{aligned}
	\partial_1 h_1&=x_1(\phi^3+\widetilde{\phi}^3)=2x_1\widetilde{\phi}^3=-2\partial_1\widetilde{\phi}\\
	x_1\partial^2_{33} h_1&=x_1(\phi^3+\widetilde{\phi}^3-3\phi^5-3\widetilde{\phi}^5)
		=2x_1(\widetilde{\phi}^3-3\widetilde{\phi}^5)
		=-2\bigl(\partial_{1}\widetilde{\phi}-\partial^2_{13}\widetilde{\phi}\bigr),
	\end{aligned}
	\qquad\textrm{for }x_3=0,
\end{equation*}
since $\phi$ and $\widetilde\phi$ coincide when $x_3=0$.
Substituting in the third boundary condition, we obtain
\begin{equation*}
	F:=2(1-\nu)\partial_3 h_3-\partial^2_{33}\beta +2(1-2\nu)\partial_1\widetilde{\phi}
		-2\partial^2_{13}\widetilde{\phi}=0\quad\textrm{for }x_3=0.
\end{equation*}
Since $\Delta F=0$, we infer
\begin{equation*}
	2(1-\nu)\partial_3 h_3-\partial^2_{33}\beta +2(1-2\nu)\partial_1\widetilde{\phi}
		-2\partial^2_{13}\widetilde{\phi}=0\qquad\textrm{for }x_3<0,
\end{equation*}
and thus, being $\partial_{1}\widetilde{\phi}=-x_1\widetilde{\phi}^3=-\partial_3(x_1\widetilde{\phi}\,\widetilde{\psi})$,
\begin{equation*}
	2(1-\nu)h_3-\partial_{3}\beta=-2x_1\widetilde{\phi}^3
		+2(1-2\nu)x_1\widetilde{\phi}\,\widetilde{\psi}\qquad\textrm{for }x_3<0,
\end{equation*}
Coupling with \eqref{firsthor}, we deduce
\begin{equation*}
	\left\{\begin{aligned}
	h_3&=-2x_1\widetilde{\phi}^3+2(1-2\nu)x_1\widetilde{\phi}\,\widetilde{\psi}\\
	\partial_3\beta&=-2(1-2\nu)x_1\widetilde{\phi}^3+2(1-2\nu)^2x_1\widetilde{\phi}\,\widetilde{\psi}\\
	\end{aligned}\right.
	\qquad\textrm{for }x_3<0.
\end{equation*}
Recalling that $\widetilde{\phi}^3=\partial_3(\widetilde{\phi}\,\widetilde{\psi})$ and
$\widetilde{\phi}\,\widetilde{\psi}=\partial_3 \widetilde{\psi}$, by integration,
\begin{equation*}
	\beta=-2(1-2\nu)x_1\widetilde{\phi}\,\widetilde{\psi}+2(1-2\nu)^2 x_1\widetilde{\psi}
	\qquad\textrm{for }x_3<0.
\end{equation*}
Using the identity $(x_3-1)\widetilde{\phi}\,\widetilde{\psi}=\widetilde{\psi}-\widetilde{\phi}$, we infer
\begin{equation*}
	x_3 h_3+\beta=x_1\bigl\{-2(1-2\nu)\widetilde{\phi}-2x_3\widetilde{\phi}^3+c_\nu\widetilde{\psi}\bigr\}.
\end{equation*}
Substituting in \eqref{vhbeta}, we get the expressions for $\mathcal{N}_{i1}$ given in \eqref{calNformulas}.
\vskip.25cm

{\bf Vertical force: $\mathcal{L}\bm{v}=\bm{e}_3\delta$.}
Choosing $h_1=h_2=0$, conditions \eqref{bdaryhbeta} become
\begin{equation*}
	\left\{\begin{aligned}
	(1-2\nu)\partial_\alpha h_3-\partial^2_{\alpha 3}\beta&=0\qquad (\alpha=1,2),\\
	2(1-\nu)\partial_3 h_3-\partial^2_{33}\beta&=0,
	\end{aligned}\right.
	\qquad\textrm{for }x_3=0.
\end{equation*}
Integrating the first relation with respect to $x_\alpha$, we obtain
\begin{equation*}
	\left\{\begin{aligned}
	(1-2\nu) h_3-\partial_{3}\beta&=0,\\
	2(1-\nu)\partial_3 h_3-\partial^2_{33}\beta&=0,
	\end{aligned}\right.
	\qquad\textrm{for }x_3=0.
\end{equation*}
Since $\Delta h_3=\Delta \beta=\delta$, identity \eqref{green} with $F:=(1-2\nu)h_3-\partial_3 \beta$ gives
\begin{equation}\label{vert1st}
	\begin{aligned}
	(1-2\nu)h_3-\partial_3 \beta&=\bigl\{(1-2\nu)G+\partial_{y_3} G\bigr\}\bigr|_{\bm{y}=-\bm{e}_3}\\
		&=(1-2\nu)(-\phi+\widetilde{\phi})-(x_3+1)\phi^3-(x_3-1)\widetilde{\phi}^3,
		\qquad\textrm{for }x_3<0.
	\end{aligned}
\end{equation}
Applying \eqref{green} to $F:=2(1-\nu)\partial_3 h_3-\partial^2_{33}\beta$, we deduce
\begin{equation*}
	\begin{aligned}
	2(1-\nu)\partial_3 h_3-\partial^2_{33}\beta
		&=\bigl\{-2(1-\nu)\partial_{y_3}G-\partial^2_{y_3y_3}G\bigr\}\bigr|_{\bm{y}=-\bm{e}_3}\\
		&=\partial_3\bigl\{-2(1-\nu)(\phi+\widetilde{\phi})+\partial_3(\phi-\widetilde{\phi})\bigr\},
		\qquad\textrm{for }x_3<0.
	\end{aligned}
\end{equation*}
Integrating with respect to $x_3$, we infer
\begin{equation*}
	2(1-\nu)h_3-\partial_{3}\beta=-2(1-\nu)(\phi+\widetilde{\phi})
		-(x_3+1)\phi^3+(x_3-1)\widetilde{\phi}^3\qquad\textrm{for }x_3<0.
\end{equation*}
Coupling with \eqref{vert1st}, we get explicit expressions for $h_3$ and $\partial_3\beta$, namely
\begin{equation*}
	\left\{\begin{aligned}
	h_3&=-\phi-(3-4\nu)\widetilde{\phi}+2(x_3-1)\tilde\phi^3,\\
	\partial_3\beta&=(x_3+1)\phi^3-c_\nu\widetilde{\phi}+(3-4\nu)(x_3-1)\widetilde{\phi}^3
	\end{aligned}\right.
	\qquad\textrm{for }x_3<0,
\end{equation*}
Differentiation of $\partial_3\beta$ with respect to $x_\alpha$ gives
\begin{equation*}
	\begin{aligned}
	\partial^2_{3\alpha}\beta&=-3x_\alpha(x_3+1)\phi^5+c_\nu x_\alpha \widetilde{\phi}^3
		-3(3-4\nu)x_\alpha(x_3-1)\widetilde{\phi}^5\\
	&=\partial_3\bigl\{x_\alpha\phi^3+c_\nu x_\alpha \widetilde{\phi}\,\widetilde{\psi}
		+(3-4\nu)x_\alpha\widetilde{\phi}^3\bigr\}
	\end{aligned}
\end{equation*}
and thus
\begin{equation*}
	\partial_{\alpha}\beta=x_\alpha\bigl\{\phi^3+c_\nu\widetilde{\phi}\,\widetilde{\psi}
		+(3-4\nu)\widetilde{\phi}^3\bigr\}\qquad\textrm{for }x_3<0,
\end{equation*}
Recalling identity \eqref{vhbeta}, we deduce the corresponding expressions
for $\mathcal{N}_{i 3}$ in \eqref{calNformulas}.
\end{proof}

The fundamental solution $\mathbf{N}=\mathbf{N}(\bm{x},\bm{y})$ in the half-space $\{x_3<0\}$ is such that
its columns $\bm{v}_1,\bm{v}_2$ and $\bm{v}_3$ solve $\mathcal{L} \bm{v}_i=\delta_{\bm{y}}\bm{e}_i$
where $\delta_{\bm{y}}$ is the Dirac delta concentrated at $\bm{y}=(y_1,y_2,y_3)$ with $y_3<0$.
Thus, the Neumann function $\mathbf{N}$ is given by
\begin{equation}\label{NcalN}
	\mathbf{N}(\bm{x},\bm{y})=\frac{1}{|y_3|}\mathbf{\mathcal{N}}\left(\frac{x_1-y_1}{|y_3|},\frac{x_2-y_2}{|y_3|},\frac{x_3}{|y_3|}\right)
\end{equation}
as a result of the homogeneity of $\delta$ and the second order degree of $\mathcal{L}$.

Recalling the definitions of $\phi, \widetilde{\phi}, \widetilde{\psi}$ and computing at $(x_1-y_1,x_2-y_2,x_3)/|y_3|$,
we obtain the identities
\begin{equation*}
	f:=-\frac{\phi}{y_3}=\frac{1}{|\bm{x}-\bm{y}|},\qquad
	\tilde f:=-\frac{\widetilde{\phi}}{y_3}=\frac{1}{|\bm{x}-\widetilde{\bm{y}}|},\qquad
	\tilde g:=-\frac{\widetilde{\psi}}{y_3}=\frac{1}{|\bm{x}-\widetilde{\bm{y}}|-x_3-y_3}.
\end{equation*}
where $\widetilde{\bm{y}}=(y_1,y_2,-y_3)$.
Hence, the components of $C_{\mu,\nu}^{-1}\mathbf{N}$ are given by
\begin{equation*}
	\begin{aligned}
	C_{\mu,\nu}^{-1}N_{\alpha \alpha}&=-(3-4\nu)f-(x_\alpha-y_\alpha)^2 f^3-\tilde f
		-(3-4\nu)(x_\alpha-y_\alpha)^2\tilde f^3-c_\nu\tilde g\\
	&\hskip4cm +c_\nu (x_\alpha-y_\alpha)^2\tilde f\tilde g^2-2x_3y_3\tilde f^3
		+6(x_\alpha-y_\alpha)^2x_3 y_3\tilde f^5\\
	C_{\mu,\nu}^{-1}N_{\alpha \beta}&=(x_\alpha-y_\alpha)(x_\beta-y_\beta)
		\bigl\{-f^3-(3-4\nu)\tilde f^3+c_\nu\tilde f\tilde g^2+6x_3y_3\tilde f^5\bigr\}\\
	C_{\mu,\nu}^{-1}N_{3 \alpha}&=(x_\alpha-y_\alpha)\bigl\{-(x_3-y_3)f^3-(3-4\nu)(x_3-y_3)\tilde f^3
		+c_\nu\tilde f\tilde g+6x_3y_3(x_3+y_3)\tilde f^5\bigr\}\\
	C_{\mu,\nu}^{-1}N_{\alpha 3}&=(x_\alpha-y_\alpha)\Bigl\{-(x_3-y_3)f^3-(3-4\nu)(x_3-y_3)\tilde f^3
		-c_\nu\tilde f\tilde g-6x_3y_3(x_3+y_3)\tilde f^5\Bigr\}\\
	C_{\mu,\nu}^{-1}N_{3 3}&=-(3-4\nu)f-(x_3-y_3)^2f^3-(1+c_\nu)\tilde f-(3-4\nu)(x_3+y_3)^2\tilde f^3\\
	&\hskip7cm +2x_3y_3\tilde f^3-6x_3y_3(x_3+y_3)^2\tilde f^5.\\
	\end{aligned}
\end{equation*}
Recollecting the expression for fundamental solution $\mathbf{\Gamma}$ in the whole space
and using the relation $\tilde f=\tilde g-(x_3+y_3)\tilde f\tilde g$, the above formulas can be rewritten as
$\mathbf{N}=\mathbf{\Gamma}+\mathbf{R}$ where $\mathbf{\Gamma}$ is computed at $\bm{x}-\bm{y}$
and the component $R_{ij}$ of $\mathbf{R}$ are given by
\begin{equation*}
	\begin{aligned}
	R_{\alpha \alpha}&=C_{\mu,\nu}\bigl\{-(\tilde f+c_\nu\tilde g)-(3-4\nu)\eta_\alpha^2\tilde f^3
		+c_\nu \eta_\alpha^2\tilde f\tilde g^2-2x_3y_3\bigl(\tilde f^3-3\eta_\alpha^2\tilde f^5\bigr)\bigr\}\\
	R_{\beta \alpha}&=C_{\mu,\nu}\eta_\alpha \eta_\beta\bigl\{-(3-4\nu)\tilde f^3
		+c_\nu \tilde f\tilde g^2+6x_3y_3\tilde f^5\bigr\}\\
	R_{3 \alpha}&=C_{\mu,\nu}\eta_\alpha\bigl\{-(3-4\nu)(\eta_3-2y_3)\tilde f^3
		+c_\nu \tilde f\tilde g+6x_3y_3 \eta_3\tilde f^5\bigr\}\\
	R_{\alpha 3}&=C_{\mu,\nu}\eta_\alpha\Bigl\{-(3-4\nu)(\eta_3-2y_3)\tilde f^3
		-c_\nu\tilde f\tilde g-6x_3y_3\eta_3\tilde f^5\Bigr\}\\
	R_{3 3}&=C_{\mu,\nu}\bigl\{-(\tilde f+c_\nu\tilde g)-(3-4\nu)\eta_3^2\tilde f^3
		+c_\nu\eta_3\tilde f\tilde g+2x_3y_3\bigl(\tilde f^3-3\eta_3^2\tilde f^5\bigr)\bigr\},
	\end{aligned}
\end{equation*}
where $\eta_\alpha=x_\alpha-y_\alpha$ for $\alpha=1,2$ and $\eta_3=x_3+y_3$,
which can be recombined as
\begin{equation*}
	\begin{aligned}
	R_{ij}&=C_{\mu,\nu}\bigl\{-(\tilde f+c_\nu\tilde g)\delta_{ij}-(3-4\nu)\eta_i\eta_j\tilde f^3
		+2(3-4\nu)y_3\bigl[\delta_{3i}(1-\delta_{3j})\eta_j+\delta_{3j}(1-\delta_{3i})\eta_i\bigr]\tilde f^3\\
	&\hskip4.5cm +c_\nu\bigl[\delta_{i3}\eta_j-\delta_{3j}(1-\delta_{3i})\eta_i\bigr]\tilde f\tilde g
		+c_\nu(1-\delta_{3j})(1-\delta_{3i})\eta_i \eta_j\tilde f\tilde g^2\\
	&\hskip7.75cm -2(1-2\delta_{3j})x_3y_3\bigl(\delta_{ij}\tilde f^3-3\eta_i\eta_j\tilde f^5\bigr)\bigr\}\\
	\end{aligned}
\end{equation*}
for $i,j=1,2,3$.
Since $x_3=\eta_3-y_3$, we obtain the decomposition $R_{ij}:=R^1_{ij}+R^2_{ij}+R^3_{ij}$
where
\begin{equation*}
	\begin{aligned}
	R^1_{ij}&:=C_{\mu,\nu}\bigl\{-(\tilde f+c_\nu\tilde g)\delta_{ij}-(3-4\nu)\eta_i\eta_j\tilde f^3\\
	&\hskip4.5cm +c_\nu\bigl[\delta_{3i}\eta_j-\delta_{3j}(1-\delta_{3i})\eta_i\bigr]\tilde f\tilde g
		+c_\nu(1-\delta_{3j})(1-\delta_{3i})\eta_i \eta_j\tilde f\tilde g^2\bigr\}\\
	R^2_{ij}&:=2C_{\mu,\nu} y_3\bigl\{(3-4\nu)\bigl[\delta_{3i}(1-\delta_{3j})\eta_j+\delta_{3j}(1-\delta_{3i})\eta_i\bigr]\tilde f^3
		-(1-2\delta_{3j})\delta_{ij}\eta_3\tilde f^3\\
	&\hskip10cm +3(1-2\delta_{3j})\eta_i\eta_j\eta_3\tilde f^5\bigr\}\\
	R^3_{ij}&:=2C_{\mu,\nu}(1-2\delta_{3j})y_3^2\bigl\{\delta_{ij} \tilde f^3-3\eta_i\eta_j\tilde f^5\bigr\}\\
	\end{aligned}
\end{equation*}
The proof of Theorem \ref{thm:fundsol} is complete.

\subsection*{Acknowledgements.}
The authors are indebted with Maurizio Battaglia for drawing the attention to the problem and for some useful
discussion on its volcanological significance.
A.Aspri, E. Beretta and C.Mascia thanks the New York University in Abu Dhabi (EAU) for its kind hospitality that permitted a further 
development of the present research.

\end{document}